\documentclass[]{interact}

\usepackage{epstopdf}
\usepackage{subcaption}
\usepackage{amsthm}
\usepackage{amsmath}
\usepackage{enumitem}
\usepackage{algorithm}
\usepackage{physics}
\usepackage{amssymb}
\usepackage{mathrsfs}
\usepackage{braket}
\usepackage{xcolor}
\usepackage[export]{adjustbox}

\usepackage{booktabs}
\usepackage{siunitx}
\usepackage{algpseudocode}
\usepackage{mathtools}
\usepackage[numbers,sort&compress]{natbib}
\bibpunct[, ]{[}{]}{,}{n}{,}{,}

\usepackage[colorlinks,citecolor=blue,linkcolor=blue]{hyperref}

\theoremstyle{plain}
\newtheorem{theorem}{Theorem}[section]
\newtheorem{lemma}[theorem]{Lemma}
\newtheorem{corollary}[theorem]{Corollary}

\newtheorem{assumption}[theorem]{Assumption}

\theoremstyle{definition}
\newtheorem{definition}[theorem]{Definition}

\theoremstyle{remark}

\newcommand{\R}{\mathbb{R}}
\newcommand{\N}{\mathbb{N}}
\renewcommand{\dd}{\mathop{}\!\mathrm{d}}
\newcommand{\specHs}{\mathbb{H}^s(\Omega)}
\newcommand{\io}{\int_{\Omega}}

\newcommand{\reference}{\mathrm{ref}}
\usepackage[pagewise]{lineno}

\begin{document}
\captionsetup[table]{position=below}
\title{Spatially sparse optimization problems in fractional order Sobolev spaces}

\author{
\name{
Anna Lentz\textsuperscript{a}\thanks{Email: anna.lentz@uni-wuerzburg.de},
Daniel Wachsmuth\textsuperscript{a}\thanks{Email: daniel.wachsmuth@uni-wuerzburg.de\\This research was partially supported bythe German Research Foundation DFG under project grant Wa 3626/5-1.
}
}
\affil{\textsuperscript{a}Institut für Mathematik, Universität Würzburg, 97074 Würzburg, Germany}
}

\maketitle

\begin{abstract}
We investigate time-dependent optimization problems in fractional Sobolev spaces with the sparsity promoting $L^p$-pseudo-norm for $0<p<1$ in the objective functional. In order to avoid computing the fractional Laplacian on the time-space cylinder $I\times \Omega$, we introduce an auxiliary function $w\in H^s(\Omega)$ that is an upper bound for the function $u\in L^2(I\times\Omega)$. We prove existence and regularity results and derive a necessary optimality condition. This is done by smoothing the $L^p$-pseudo-norm and by penalizing the inequality constraint regarding $u$ and $w$. The problem is solved numerically with an iterative scheme whose weak limit points satisfy a weaker form of the necessary optimality condition. \end{abstract}

\begin{amscode}
49K30, 
49M20  
\end{amscode}

\section{Introduction}

The aim of this paper is to investigate time-dependent optimization or optimal control problems
that admit spatially sparse solutions.
Let us  describe informally the notion of spatial sparsity.
Let $\Omega \subset \R^d$ be a bounded domain and $T>0$ a terminal time, and set $I:=(0,T)$.
Given a function $u:I \times \Omega \to \R$,
we say that $u$ is spatially sparse if the measure of the set $\{x\in \Omega: \ u(\cdot, x) \ne 0\}$ is small.

In this paper, we will analyse the following optimization problem:
\begin{equation} \label{eq:minprob}
    \underset{u\in L^2(I\times\Omega), w\in H^s(\Omega)}{\min} f(u) + \frac{\alpha}{2}\|u\|_{L^2(I\times\Omega)}^2+\frac{\beta}{2}\|w\|_{H^s(\Omega)}^2+\gamma \|w\|_p^p
\end{equation}
subject to the constraints
\begin{equation} \label{eq:conduw}
    |u(t,x)| \leq w(x) \quad  \textrm{f.a.a.}\,\, (t,x) \in I\times\Omega.
\end{equation}
In this problem, $u$ is the unknown control that should be spatially sparse, while $w$ is an auxiliary variable.
Let us first explain the various parts of the optimization problem in more detail:
\begin{itemize}
 \item $f: L^2(I\times\Omega) \to \R$ is the objective functional to be minimized,
 \item the term $\frac{\alpha}{2}\|u\|_{L^2(I\times\Omega)}^2$ with $\alpha>0$
 models the control cost of $u$,
 \item the $L^p$-pseudo-norm $\gamma \|w\|_p^p:= \gamma \int_\Omega |w(x)|^p \dd x$ with $p\in (0,1)$ and $\gamma>0$ promotes sparsity of $w$,
 \item the term $\frac{\beta}{2}\|w\|_{H^s(\Omega)}^2$ with $s\in (0,1)$ and $\beta>0$ is needed to prove existence of solutions,
 \item the constraint \eqref{eq:conduw} couples $u$ and $w$: if $w$ is sparse, i.e. the measure of $\{x\in \Omega: \ w(x)\ne0\}$ is small, then $u$ is spatially sparse.
\end{itemize}
The main results of this paper regarding \eqref{eq:minprob}--\eqref{eq:conduw} are: existence of solutions, first-order necessary optimality conditions,
and a solution algorithm. The optimality conditions and the algorithm are both based on the same smoothing procedure.

To motivate the form of the functional in \eqref{eq:minprob} let us first consider the problem with $\beta=0$. Then we can eliminate $w$ by setting
$w(x):= \|u(\cdot,x)\|_{L^\infty(I)}$, and we obtain the problem
\begin{equation} \label{eq:minprob_beta0}
    \underset{u\in L^2(I\times\Omega)}{\min} f(u) + \frac{\alpha}{2}\|u\|_{L^2(I\times\Omega)}^2 + \gamma \int_\Omega \|u(\cdot,x)\|_{L^\infty(I)}^p \dd x.
\end{equation}
For suitable functions $u$ one can prove that for $p \searrow0$ we have the convergence
\begin{equation} \label{eq_p_to_0}
 \int_\Omega \|u(\cdot,x)\|_{L^\infty(I)}^p \dd x \to \int_\Omega \|u(\cdot,x)\|_{L^\infty(I)}^0 \dd x = \left| \{x\in \Omega: \ u(\cdot, x) \ne 0\} \right|
\end{equation}
with the convention $0^0:=0$,
where $|\cdot|$ denotes the Lebesgue measure.
Note that in \eqref{eq_p_to_0} we can replace the $L^\infty(I)$-norm by any $L^q(I)$-norm with $q\in [1,+\infty)$.
Hence, we expect solutions of \eqref{eq:minprob_beta0} or \eqref{eq:minprob}--\eqref{eq:conduw} to be spatially sparse for $0 <p \ll 1$.
The choice $p<1$ leads to a functional that is not weakly lower semicontinuous in $L^q$-spaces, so that standard existence proofs do not work.

In order to guarantee existence of solutions, one needs to set the problem \eqref{eq:minprob_beta0} in a function space that compactly embeds into $L^2(I\times\Omega)$.
Here, we will use fractional Sobolev spaces $H^s$ with $0<s \ll 1$.
Then a regularized version of  \eqref{eq:minprob_beta0} reads
\begin{equation} \label{eq_minprob_u_in_Hs}
    \underset{u\in L^2(I\times\Omega)}{\min} f(u) + \frac{\alpha}{2}\|u\|_{L^2(I\times\Omega)}^2 + \gamma \int_\Omega \|u(\cdot,x)\|_{L^2(I)}^p \dd x + \beta \|u\|_{H^s(I\times\Omega)}^2.
\end{equation}
Due to compact embeddings of $H^s(I\times \Omega)$ into $L^2(I\times \Omega)$, one can prove existence of solutions of this problem.
To solve the resulting problem numerically, we want to use the majorization-minimization approach from \cite{lp_cont,sparse}.
When applied to \eqref{eq_minprob_u_in_Hs}, one would have to solve equations of the type
$(-\Delta)^s u + au = f$ in each iteration, where $(-\Delta)^s$ is the fractional Laplace on the time-space domain $I\times \Omega$ and $a\in L^\infty( I \times \Omega)$
is  a non-constant function that changes during the iterations. Surprisingly, no efficient solution methods are known for these type of equations.
Spectral or Fourier based approaches are only viable if $a\in \R$, since in this case the standard Laplacian $-\Delta$ and the operator $(-\Delta)^s  + a$ have the same system of eigenfunctions.
A direct finite-element discretization with mesh-size $h$ would lead to a linear system with $h^{-(d+1)}$ unknowns and full system matrix,
whose solution would be prohibitively expensive already for $d=2$.
The Caffarelli-Silvestre extension replaces the fractional Laplace equation by a degenerate elliptic problem in $I\times \Omega \times (0,+\infty) \subset \R^{d+2}$, see \cite{StingaTorrea2010},
which again is challenging to solve for $d\ge2$.

To overcome this difficulty, we introduce the auxiliary variable $w \in H^s(\Omega)$. In the resulting optimization problem the sparsity promoting functional is applied to $w$,
and the condition $|u(t,x)| \le w(x)$ enforces spatial sparsity of $u$. Still we need to solve equations involving the fractional Laplacian on a subset of $\R^d$ instead
on a subset of $\R^{d+1}$.
For $s<1/2$ the space $H^s(\Omega)$ contains characteristic functions, so we later work with this choice of $s$.
These considerations lead us to the optimization \eqref{eq:minprob}--\eqref{eq:conduw} introduced at the beginning.

Let us provide a review of relevant literature.
Actuator placement problems are studied for instance in \cite{Munch2008,KaliseKunischSturm2018},
where the $L^0$-pseudo-norm appears in the functional, and topological gradient methods are used to solve the problem numerically.
Problems with $p\in[0,1)$ are studied for instance in  \cite{Wachsmuth2019,sparse,lp_cont}.
A simple example without solution is presented in \cite{Wachsmuth2019}.
The regularization in $H^1$ was proposed in \cite{lp_cont}, while fractional Sobolev spaces are used in \cite{sparse}.
For $p\in (0,1)$, majorization-minimization methods are available that are based on approximations of the concave function $s \mapsto s^{p/2}$.
This approach was used in \cite{lp_cont,sparse}, and we will use these ideas as well.
The present paper is a direct extension of \cite{sparse} to the time-dependent case. Due to the presence of the inequality constraints \eqref{eq:conduw}
the analysis is considerably more technical.

The choice $p=1$ is very popular in the literature, since then tools from convex analysis can be applied.
Optimal control problems that penalize the $L^1$-norm of the controls in the cost functionals were analysed, e.g., in \cite{180NavierStokesOCExample, 181SchloeglExample,85l1Sparse}.
Problems with spatially sparse controls were considered in \cite{81spatTempSparse,80dirSparse}.
As explained above, we concentrate on  the case $p<1$ in this paper to stay close to the $L^0$-pseudo norm.
For the numerical experiments, we used standard finite elements. More advanced concepts are available to solve the equations of the type $(-\Delta)^s u = f$ numerically.
We refer only to the review article \cite{Hofreither2020}. However, these methods do not generalize easily to equations of the type $(-\Delta)^s u +au= f$ with non-constant $a$.

The structure of the paper is as follows. First, we give a brief overview of fractional Sobolev spaces in Section \ref{sec:fracSob}. Then we prove existence of solutions of the minimization problem in Section \ref{sec:exSol} and derive an auxiliary problem with a smoothing and penalization scheme in the subsequent Section \ref{sec:auxProb}. This auxiliary problem is used in Section \ref{sec:oc} to obtain an optimality condition for the minimization problem \eqref{eq:minprob}. In Section \ref{sec:algo} we consider an iterative scheme to solve this minimization problem. We conclude by giving some numerical examples in Section \ref{sec:numRes}.

\subsection*{Notation}
Throughout this paper, $\Omega\subset \R^d$ denotes a bounded Lipschitz domain and $I=(0,T)$ some time interval with $T>0$. Furthermore, we use the abbreviation
\[
U:= L^2(I\times\Omega).
\]
The space $W$ will be defined below as a suitable fractional Sobolev space on $\Omega$. We denote by
\begin{equation}\label{eq_def_K}
K:= \{ (u,w) \in U\times W:\ |u(t,x)| \leq w(x) \textrm{ for a.a. } (t,x) \in I\times\Omega \}
\end{equation}
the set of admissible functions. Moreover, we use the notation $(\cdot,\cdot)_V$ for the inner product of a Hilbert space $V$ and $\braket{\cdot,\cdot}_V$ for its duality product. We use the notation
\[
u_- := \min(u,0) \quad \text{ and }\quad u_+ :=\max(u,0)
\]
to describe the negative or positive part of a function $u$.

When we use a function $w$ defined on the spatial domain $\Omega$ in an expression involving the time-space cylinder $I\times\Omega$,
we employ the constant extension of $w$ from $\Omega$ to $I\times\Omega$ with $w(t,x):= w(x)$ for (almost all) $(t,x) \in I\times\Omega$.

\section{Fractional Sobolev spaces}\label{sec:fracSob}
In this section, we state some definitions and results about fractional order Sobolev spaces.
There are several possibilities to define a fractional Sobolev space or a fractional Laplace operator on a bounded domain.
For most cases, they lead to to the same spaces with equivalent norms, see \eqref{eq:equiv_spaces} below.
As the solutions of optimization problems as \eqref{eq:minprob} depend on the concrete choice of the norm,
we will discuss several of them.
Important for our analysis are two facts: the compact embedding into $L^2(\Omega)$ and a certain monotonicity of the inner product of the fractional Sobolev space with respect to
truncated functions, see Lemma \ref{tm:propFrac} and \ref{lm:inequPosPart} below.
These are satisfied for the standard choices of fractional Sobolev spaces.

\begin{definition} Let $s\in (0,1)$.
Then the fractional order Sobolev space $H^s(\Omega)$
is defined as
\[
H^{s}(\Omega) :=  \left\{ w \in L^2(\Omega): \ (x,y)\mapsto\frac{|w(x)-w(y)|}{|x-y|^{\frac{d}{2}+s}} \in L^2(\Omega \times \Omega) \right\}
\]
 with the corresponding norm \begin{equation*}
 \|w\|_{H^{s}(\Omega)} := \left(\int_{\Omega}|w|^2\dd x+\frac{c_{d,s}}{2}\int_{\Omega}\int_{\Omega} \frac{|w(x)-w(y)|^2}{|x-y|^{d+2s}}\dd y \dd x \right)^{\frac12}
  \end{equation*}
and inner product
\[
	(u,w)_{H^s(\Omega)}  = \io uw \dd x + \frac{c_{d,s}}{2} \io \io \frac{(u(x)-u(y)) (w(x)-w(y)))}{|x-y|^{d+2s}} \dd y \dd x,
\] 
where $c_{d,s}:= \frac{s 2^{2s}\Gamma(s+\frac{d}{2})}{\pi^{\frac{d}{2}}\Gamma(1-s)}.$
 \end{definition}

With this norm, we define the space $H_0^s(\Omega)$ as $$ H_0^s(\Omega) := \overline{C_c^{\infty}(\Omega)}^{H^s(\Omega)}, $$ where $C_c^{\infty}(\Omega)$ is the space of infinitely continuously differentiable functions with compact support in $\Omega$.
Moreover, let $$ \tilde{H}^s(\Omega) := \left\{ w \in H^s(\R^d):  w|_{\R^d\setminus  \Omega} = 0 \right\}=\overline{C_c^{\infty}(\Omega)}^{H^s(\R^d)},$$
where the latter identity was shown in \cite{7densityFrac}.
Hence, functions in $\tilde{H}^s(\Omega)$ can be extended by zero to $\R^d \setminus \Omega$.
The space $\tilde{H}^s(\Omega)$ is supplied with the $H^s(\R^d)$-inner product.
For $s=\frac12$, the Lions-Magenes space is defined as
\[
H_{00}^{\frac12}(\Omega):= \left\{w \in L^2(\Omega): \int_{\Omega} \frac{w^2(x)}{\operatorname{dist}(x,\partial \Omega)}\dd x < \infty \right\}
\]
with the norm
\[
\|w\|_{H_{00}^{\frac12}(\Omega)} := \left(\|w\|^2_{H^{\frac12}(\Omega)}+ \int_{\Omega} \frac{w^2(x)}{\operatorname{dist}(x,\partial \Omega)}\dd x\right)^{\frac12}.
\]
For this space it holds $H_{00}^{\frac12}(\Omega) \subsetneq H^{\frac12}(\Omega)$.

Another way to define fractional Sobolev spaces is via spectral theory using eigenvectors and eigenfunctions of the Laplace operator with zero Dirichlet boundary conditions.
\begin{definition} Let $s\geq 0$.
Then the fractional order Sobolev space in spectral form is defined as  $$ \mathbb{H}^s(\Omega) := \left\{ w= \sum_{n=1}^{\infty} w_n \phi_n \in L^2(\Omega): \|w\|_{\mathbb{H}^s}^2 := \sum_{n=1}^{\infty} \lambda_n^s w_n^2<\infty \right\},$$
where $w_n := \int_{\Omega}w \phi_n \dd x$, and $\phi_n$ are the eigenfunctions of the Laplacian with zero Dirichlet boundary conditions to the eigenvalues $\lambda_n$.
\end{definition}
One can rewrite the norm of $\mathbb{H}^s(\Omega)$ in an integral formulation as done in  \cite[page 1, Appendix A.1]{nonhomBC}, \cite[page 5-6]{anote}, namely \begin{equation} \label{eq:spectralInt}\|w\|^2_{\mathbb{H}^s(\Omega)}= \frac12 \int_{\Omega}\int_{\Omega}|(w(x)-w(y))|^2J(x,y) \dd y \dd x + \int_{\Omega}\kappa(x) |w(x)|^2\dd x\end{equation}
with inner product
\[
	(u,w)_{\specHs} = \frac12 \io \io (u(x)-u(y)) (w(x)-w(y)) J(x,y) \dd y \dd x + \io \kappa(x) uw \dd x,
\]
where $J$ and $\kappa$ are measurable non-negative functions, and $J$ is symmetric with $J(x,y)=J(y,x)$  for a.a. $x,y \in \Omega$.
This formulation of the $\specHs$-inner product as a double integral will be used in the proof of Lemma \ref{lm:inequPosPart}. 

It was shown in \cite[Corollary 1.4.4.5]{105ellipticProb} and \cite[Chapter 1]{ lions} or \cite[page 12]{185fracEllipticEq} that
we have the equalities
\begin{align} \label{eq:equiv_spaces}\mathbb{H}^s(\Omega)  = \begin{cases} H^s(\Omega) = H_0^s(\Omega)= \tilde{H}^s(\Omega) & \textrm{  if  }0<s<\frac12,  \\  H_{00}^{1/2}(\Omega) & \textrm{  if  } s=\frac12, \\ H_0^s(\Omega)=\tilde{H}^s(\Omega)  & \textrm{  if  } \frac12<s<1,  \end{cases} \end{align}
with equivalence of norms.
From now on, the space $W$ refers to a fractional Sobolev space that is covered by the previous result, i.e.,
\begin{equation}\label{eq_def_W}
 W \in \{ \ H^s(\Omega) , \  H_0^s(\Omega), \   \tilde{H}^s(\Omega), \  \mathbb{H}^s(\Omega) \  \}
\end{equation}
for some $s\in (0,1)$.
We will work with the space $W$ in the remainder of this paper. While all the spaces in \eqref{eq_def_W}
share the same analytical properties, the solutions of optimization problems depend on the concrete choice of the space.

The following properties of these fractional Sobolev spaces will be used later on.
\begin{lemma} \label{tm:propFrac}
    \begin{enumerate}[label=(\roman*)]
        \item The embedding $W\hookrightarrow L^2(\Omega)$ is compact.
        \item Let $w\in W$ and let $z(x):=\max(-1,\min(w(x),1))$. Then $z\in W$ and $(w,z)_W\geq \|z\|_W^2$.
        \item $C_0(\Omega) \cap W$ is dense in $W$ and $C_0(\Omega)$.
    \item There is $q>2$ such that $W$ is continuously embedded in $L^q(\Omega)$. More precisely, \begin{align*}
        W \hookrightarrow \begin{cases}
            L^{\frac{2d}{d-2s}}(\Omega) & \textrm{if } d>2s, \\
            L^p(\Omega) \quad  \forall p\in[1,\infty) & \textrm{if } d=2s, \\
            C^{0,s-\frac{d}{2}}(\Bar{\Omega}) & \textrm{if } d<2s.
        \end{cases}
    \end{align*}
    \end{enumerate}
\end{lemma}
\begin{proof}
    This is shown in \cite[Section 6]{sparse} and \cite[Sections 6, 7]{hitch}.
\end{proof}

\begin{lemma}\label{lm:inequPosPart}
Let $W$ as in \eqref{eq_def_W} and let $w\in W$. Then $(w,w_+)_W \le \|w\|_W^2$.
\end{lemma}
\begin{proof}
    One can directly see that the inequality holds for the $L^2(\Omega)$-part in the $W$-norm, i.e., $(w,w_+)_{L^2(\Omega)} \le \|w\|_{L^2(\Omega)}^2$.
    Let $w: \Omega \to \R$. Then it holds
    
    \begin{equation} \label{eq_w_pos_part}
 (w_+(x)-w_+(y))\cdot (w(x)-w(y))\leq (w(x)-w(y))^2
 \end{equation}
     for a.a.  $x,y \in \Omega$. This can be seen by considering the following cases:
    \begin{enumerate}
        \item If $w(x)\ge 0$ and $w(y)\ge 0$, it holds  $(w_+(x)-w_+(y))\cdot (w(x)-w(y)) = (w(x)-w(y))^2.$
        \item For $w(x)\le 0$ and $w(y)\le 0$ we get $(w_+(x)-w_+(y))\cdot (w(x)-w(y)) = 0 \le (w(x)-w(y))^2.$
        \item If $w(x)\ge 0$ and $w(y)\le 0$, it holds  $(w_+(x)-w_+(y))\cdot (w(x)-w(y)) = (w_+(x))\cdot (w(x)-w(y)) \le (w(x)-w(y))^2.$
        \item Similarly, for $w(x)\le 0$ and $w(y)\ge 0$ we obtain $(w_+(x)-w_+(y))\cdot (w(x)-w(y)) = (-w_+(y))\cdot (w(x)-w(y)) \le (w(x)-w(y))^2.$
    \end{enumerate}
    Analogously, we can prove \eqref{eq_w_pos_part} for all $w: \R^d\to \R$ and all $x,y \in \R^d$.
    Multiplying the inequality \eqref{eq_w_pos_part} with the kernels and integrating on $\Omega\times \Omega$ or $\R^d\times \R^d$ proves
    the corresponding inequality for the double integrals.
    By definition of the inner product in $W$ this yields the result, where for the spectral fractional Sobolev space $W=\mathbb{H}^s(\Omega)$,
    we used the integral formulation \eqref{eq:spectralInt}.
\end{proof}

\section{Existence of Solutions}\label{sec:exSol}
We now want to prove existence of solutions of the above introduced minimization problem
\begin{align*}
    \underset{(u,w)\in K}{\min} f(u) + \frac{\alpha}{2}\|u\|_U^2+\frac{\beta}{2}\|w\|_W^2+\gamma \int_{\Omega}|w|^p\dd x.
\end{align*}
Let us recall the definition of $U = L^2(I \times \Omega)$ and $W$ in \eqref{eq_def_W}.
Note that the set $K$ of admissible functions defined in \eqref{eq_def_K} is convex.
First, we state some assumptions that should hold throughout the paper.
\begin{assumption}\label{ass:f}
\begin{enumerate}
    \item The function $f\colon U\to\R$ is weakly lower semicontinuous and bounded from below by an affine function, i.e., there are $g\in U^*$, $c\in \R$ such that $f(u) \geq g(u) + c$ for all $u\in U$.
    \item Furthermore, $f\colon U\to\R$ is continuously Fréchet differentiable.
    \item The constants $\alpha$, $\beta$, $\gamma$ are positive, and $p\in (0,1)$.
\end{enumerate}
\end{assumption}
In this paper, we identify $U=L^2(I \times \Omega)$ with its dual space.
Accordingly, $f'(u)\in U^*$ is identified with its Riesz representative $f'(\bar u)\in U$.

\begin{theorem}\label{tm:exsol}
Under Assumption \eqref{ass:f}, the minimization problem \eqref{eq:minprob}--\eqref{eq:conduw} admits a solution.
\end{theorem}
\begin{proof}
This statement can be shown with standard arguments. Let $(u_n,w_n) $ be a minimizing sequence. Then Assumption \ref{ass:f} yields boundedness of $(u_n)$ in $U$ and also of $(f(u_n) + \frac{\alpha}{2}\|u_n\|_U^2)$, which leads to boundedness of $(w_n)$ in $W$. Hence, after possibly extracting a subsequence there are $\bar{u} \in U$ and $\bar{w}\in W$ such that $u_n \rightharpoonup \bar u$ in $U$ and $w_n \rightharpoonup \bar{w}$ in $W$ with $w_n \to \bar{w}$ in $L^2(\Omega)$.
This implies $\gamma \int_{\Omega}|w_n|^p\dd x \to \gamma \int_{\Omega}|\bar w|^p\dd x$, see \cite[Lemma 5.1]{lp_cont}.
Since $K$ is closed and convex and therefore weakly closed, it holds $(\bar u,\bar{w})\in K$. Thus, \begin{align*}
     \liminf_{n \to \infty} f(u_n) + \frac{\alpha}{2}\|u_n\|_U^2+\frac{\beta}{2}\|w_n\|_W^2+\gamma \int_{\Omega}|w_n|^p\dd x \\ \geq f(\bar u) + \frac{\alpha}{2}\|\bar u\|_U^2+\frac{\beta}{2}\|\Bar{w}\|_W^2+\gamma \int_{\Omega}|\bar{w}|^p\dd x
\end{align*} due to weak lower semicontinuity.
This shows that $(\bar u,\bar{w})$ is a solution of \eqref{eq:minprob}.
\end{proof}

\section{Necessary optimality condition}

The aim of this section is to derive a necessary optimality condition for problem \eqref{eq:minprob}.
Throughout this section, $(\bar u, \Bar{w})$ will be a local solution of the non-smooth problem \eqref{eq:minprob}.
Since the $L^p$-pseudo-norm is not differentiable, we employ a smoothing scheme.

\subsection{Smooth auxiliary problem} \label{sec:auxProb}
Let us construct a sequence of auxiliary minimization problems that are smooth enough to directly obtain necessary optimality conditions. These optimality conditions can then be used to derive an optimality condition for the original problem \eqref{eq:minprob}.

 To do so, we use a smoothing scheme as in \cite{sparse, lp_cont}. There, a smooth approximation of $t\mapsto t^{p/2}$ is introduced by
 \begin{equation*}
     \psi_{\epsilon}(t) := \begin{cases}
          \frac{p}{2}\frac{t}{\epsilon^{2-p}}+(1-\frac{p}{2})\epsilon^p & \textrm{if } t \in [0,\epsilon^2), \\ t^{p/2} & \textrm{if } t \geq \epsilon^2,
     \end{cases}
 \end{equation*}
 for some $\epsilon\ge 0$.
 Then \begin{align*}
     G_{\epsilon}(w):= \int_{\Omega} \psi_{\epsilon}(|w|^2) \dd x
 \end{align*} is an approximation of $\int_{\Omega}|w|^p \dd x$. Note that it holds $G_0(w) = \int_{\Omega}|w|^p \dd x$.

For the convergence analysis later on we use the following results of \cite{sparse}.

\begin{lemma}
{\normalfont(\cite[Lemma 4.1, 4.2, 4.3]{sparse})}
\label{tm:psieps} Let $p\in(0,2)$. Then $\psi_{\epsilon}: [0, \infty) \to [0,\infty)$
has the following properties:
\begin{enumerate}[label=(\roman*)]
\item The function $\psi_{\epsilon}$ is concave.
\item $\epsilon \mapsto \psi_{\epsilon}(t)$ is monotonically increasing for all $t\geq 0$.
\item For sequences $(w_k)$, $(\epsilon_k)$ such that $w_k \to w$ in $L^1(\Omega)$ and $\epsilon_k \to \epsilon \geq 0$ it holds $G_{\epsilon_k}(w_k) \to G_{\epsilon}(w)$.
\item Let $\epsilon>0$. Then $\psi_{\epsilon}$ is continuously differentiable with derivative \begin{align*}
	\psi_{\epsilon}'(t)=\frac{p}{2}\min(\epsilon^{p-2},t^{\frac{p-2}{2}}),
\end{align*} and $w \mapsto G_{\epsilon}(w)$ is Fréchet differentiable from $L^2(\Omega)$ to $\R$ with \begin{align*}
G_{\epsilon}'(w)h = \int_{\Omega} 2 w(x) \psi_{\epsilon}'(w(x)^2) h(x) \dd x.
\end{align*}
\end{enumerate}

\end{lemma}
We remark that while $\psi_{\epsilon}$ is concave, this is not the case for the composition $t\mapsto \psi_{\epsilon}(t^2)$ that we use in the definition of $G$ to approximate the $L^p$-pseudo-norm.\\
Now we replace the $L^p$-pseudo-norm by its smooth approximation to obtain the auxiliary objective function \begin{align} \label{eq:phi_def}
    \Phi_{\epsilon}(u,w) := f(u) + \frac{\alpha}{2}\|u\|_U^2+\frac{\beta}{2}\|w\|_W^2+\gamma G_{\epsilon}(w).
\end{align}
For $\epsilon=0$, the function $\Phi_0$ reduces to the original objective.

 Next, we also penalize the violation of the constraint $|u|\leq w$ in the original minimization problem.  To do so, we relax problem \eqref{eq:minprob} to
\begin{equation*}
    \underset{(u,w) \in U\times W}{\min} \Phi_{\epsilon}(u,w) +  \frac1{2\epsilon} \| (w-u)_-\|_U^2 +\frac1{2\epsilon} \| (w+u)_-\|_U^2,
\end{equation*}
adding the $L^2(I\times\Omega)$-norm of the parts that violate the constraint and multiplying them with a penalization parameter $\frac1{\epsilon}$.
Here, we use the notation
\[
 \| (w-u)_-\|_U^2 := \int_{\Omega}\int_I ((w(x)-u(t,x))_-)^2 \dd t \dd x,
\]
and similarly for $\| (w+u)_-\|_U^2$,
using the constant extension of $w$ from $\Omega$ to $I \times \Omega$.

Let $(\bar u, \Bar{w})$ be a local  solution of the non-smooth problem \eqref{eq:minprob}.
Then there is $\rho>0$ such that $\Phi_0(\bar u,\Bar{w}) \leq \Phi_0(u,w)$ for all $(u,w)$ with $\|u-\bar u\|_{U} \leq \rho, \|w-\Bar{w}\|_{W} \leq \rho$ and $(u,w)\in K$.

In order to enforce convergence of solutions of relaxed problems to $(\bar u, \bar w)$, we
consider the auxiliary problem
\begin{equation} \begin{gathered}\label{eq:auxprob}
    \underset{(u,w) \in U\times W}{\min} \Phi_{\epsilon}(u,w) +  \frac12 \|u-\bar u\|^2_{U} + \frac12 \|w-\Bar{w}\|_{L^2(\Omega)}^2
    + \frac1{2\epsilon} \| (w-u)_-\|_U^2 +\frac1{2\epsilon} \| (w+u)_-\|_U^2    \\
    \textrm{subject to }  \|u-\bar u\|_{U} \leq \rho, \|w-\Bar{w}\|_{W} \leq \rho  .
    \end{gathered}
\end{equation}
Existence of solution of this problem can be proven by standard arguments similarly as in Theorem \ref{tm:exsol}.

Note that the penalization is an additional step compared to the procedure in \cite{sparse}, where the considered minimization problem is unconstrained.
Without the penalization, we were not able to prove boundedness of both the Lagrange multipliers to the inequality constraints \eqref{eq:conduw} and the derivatives of the smooth approximation $G_{\epsilon}$ in the optimality conditions of a sequence of auxiliary problems \eqref{eq:auxprob}.
The penalization enables us to obtain boundedness results for approximations of the multipliers to the inequality constraints, see  Lemma \ref{tm:mult_bounded} below. This then implies also boundedness of the derivatives of $G_{\epsilon}$.

\subsection{Passing to the limit}\label{sec:oc}
In order to derive an optimality condition for problem \eqref{eq:minprob}, we pass to the limit $\epsilon\searrow 0$ in the optimality condition of problem \eqref{eq:auxprob}. First, we show that solutions of \eqref{eq:auxprob} converge to the reference solution $(\bar u,\bar w)$ for $\epsilon \searrow 0$. We start with an auxiliary lemma.

\begin{lemma}\label{tm:auxLimSupInf}
Let $N\in \N$.
Let $(a^1_k), \dots, (a^N_k)$ be  sequences with
\[
 \liminf_{k \to \infty} a_k^i \in \R \quad \forall i=1,\dots, N
\]
and
\[
\limsup_{k\to\infty} \left(\sum_{i=1}^N a^i_k \right) \le \sum_{i=1}^N ( \liminf_{k\to\infty} a^i_k ).
\]
Then all sequences $(a^1_k), ..., (a^N_k)$ are convergent with limits in $\R$.
\end{lemma}
\begin{proof}
The claim is elementary for $N=1$. Let the claim be proven for some $N\ge1$. Consider $N+1$ sequences satisfying the assumptions of the lemma.
Then
\begin{equation}\label{eq:eqInAuxLimLemma}
 \limsup_{k\to\infty} \left(\sum_{i=1}^N a^i_k \right) + \liminf_{k\to\infty} a^{N+1}_k
 \le \limsup_{k\to\infty} \left(\sum_{i=1}^{N+1} a^i_k \right)
 \le \sum_{i=1}^{N+1} ( \liminf_{k\to\infty} a^i_k ).
\end{equation}
This implies $\limsup_{k\to\infty} \left(\sum_{i=1}^N a^i_k \right) \le \sum_{i=1}^N ( \liminf_{k\to\infty} a^i_k )$, and hence the sequences $(a^1_k), ..., (a^N_k)$
converge by induction hypothesis.
Cancelling these convergent sequences  in \eqref{eq:eqInAuxLimLemma} we obtain $\limsup_{k\to\infty}a^{N+1}_k = \liminf_{k\to\infty} a^{N+1}_k$, so $(a^{N+1}_k)$ converges as well.

\end{proof}

\begin{lemma}\label{lem_aux_strong_convergence}
  Let $\epsilon_k \searrow 0$ and let $(u_k,w_k)$ be the solution of \eqref{eq:auxprob} for the corresponding smoothing parameter $\epsilon_k$.
  Suppose $u_k\rightharpoonup u^*$ in $U$, $w_k \rightharpoonup w^*$ in $W$,
  then
  \[
   \liminf_{k\to\infty} \Phi_{\epsilon_k} (u_k,w_k) \ge \Phi_0(u^*,w^*).
  \]
  If in addition
  \begin{equation}\label{eq:asslimsupPhi0}
  \Phi_0(u^*,w^*) \ge \limsup_{k \to \infty} \Phi_{\epsilon_k}(u_k,w_k)
  \end{equation}
  is satisfied, then $u_k \to u^*$ in $U$ and $w_k \to w^*$ in $W$.
\end{lemma}
\begin{proof}
 By Lemma \ref{tm:psieps} (iii), it holds ${G_{\epsilon_k}(w_k) \to G_0(w^*)}$, which implies
 \begin{multline*}
 \liminf_{k\to\infty} \Phi_{\epsilon_k} (u_k,w_k)
 \\
 \begin{aligned} &\ge \liminf_{k\to\infty}f(u_k)+ \liminf_{k\to\infty}\frac{\alpha}{2}\|u_k\|_U^2+ \liminf_{k\to\infty} \frac{\beta}{2}\|w_k\|_W^2+ \liminf_{k\to\infty} \gamma G_{\epsilon}(w_k) \\
 &\ge \Phi_0(u^*,w^*)
\end{aligned}
\end{multline*}
 by weak lower semicontinuity of $f$ and the squared norms on $U$ and $W$. This shows the first part.
 For the second statement, we want to apply Lemma \ref{tm:auxLimSupInf} to the terms in $\Phi_{\epsilon_k}(u_k,w_k)$. The inequality above together with assumption \eqref{eq:asslimsupPhi0} yields
 \[
 \liminf_{k\to\infty}f(u_k)+ \liminf_{k\to\infty}\frac{\alpha}{2}\|u_k\|_U^2+ \liminf_{k\to\infty} \frac{\beta}{2}\|w_k\|_W^2+ \liminf_{k\to\infty} \gamma G_{\epsilon_k}(w_k) \ge \limsup_{k \to \infty} \Phi_{\epsilon_k}(u_k,w_k)
 ,
 \]
 and the assumptions of Lemma \ref{tm:auxLimSupInf} are satisfied.
 Therefore, we obtain convergence of all the terms in $\Phi_{\epsilon_k}(u_k,w_k)$.
 In particular, it holds $\|u_k\|_U \to \|u^*\|_U$ and $\|w_k\|_W\to \|w^*\|_W$ which together with weak convergence of $(u_k)$ and $(w_k)$ leads to the desired result.
\end{proof}

\begin{theorem} \label{tm:conv_uw}
    Let $\epsilon_k \searrow 0$, and let $(u_k,w_k)$ be solutions of \eqref{eq:auxprob} for the corresponding smoothing parameter $\epsilon_k$. Then $u_k \to \bar u$ in $U$ and $w_k \to \Bar{w}$ in $W$.
\end{theorem}
\begin{proof}
    The sequences $(u_k)$ and $(w_k)$ are bounded in $U$ and $W$, respectively, due to the constraints in \eqref{eq:auxprob}.
    Hence, there are $u^* \in U$, $w^* \in W$ such that (after extracting a subsequence if necessary)  $u_k\rightharpoonup u^*$ in $U$, $w_k \rightharpoonup w^*$ in $W$, and $w_k \to w^*$ in $L^2(\Omega)$.
    Using Lemma \ref{tm:psieps} (iii), it holds $G_{\epsilon_k}(\Bar{w})\to G_0(\Bar{w})$, so that $\Phi_{\epsilon_k}(\bar u,\Bar{w}) \to \Phi_0(\bar u,\Bar{w})$ follows.
    From optimality of $(u_k,w_k)$ in \eqref{eq:auxprob} we get
    \begin{multline*}
        \Phi_{\epsilon_k}(\bar u,\Bar{w})\geq \Phi_{\epsilon_k}(u_k,w_k) +  \frac12 \|u_k-\bar u\|^2_{U} + \frac12 \|w_k-\Bar{w}\|_{L^2(\Omega)}^2 \\
        + \frac1{2\epsilon_k} \| (w_k-u_k)_-\|_U^2 +\frac1{2\epsilon_k} \| (w_k+u_k)_-\|_U^2 .
    \end{multline*}
    Passing to the limit superior and using the first claim of Lemma \ref{lem_aux_strong_convergence} leads to
     \begin{align}\begin{split}\label{eq:chainInequLimsup}
        \Phi_0(\bar u,\Bar{w})
        &\geq \limsup_{k \to \infty} \left( \Phi_{\epsilon_k}(u_k,w_k) + \frac12 \|u_k-\bar u\|^2_{U} + \frac12 \|w_k-\Bar{w}\|_{L^2(\Omega)}^2 \right. \\
        &\quad + \left.\frac1{2\epsilon_k} \| (w_k-u_k)_-\|_U^2 +\frac1{2\epsilon_k} \| (w_k+u_k)_-\|_U^2   \right)  \\
        &\geq  \liminf_{k \to \infty} \Phi_{\epsilon_k}(u_k,w_k) +  \liminf_{k \to \infty}\frac12 \|u_k-\bar u\|^2_{U} + \liminf_{k \to \infty} \frac12 \|w_k-\Bar{w}\|_{L^2(\Omega)}^2 \\
         &\quad
         + \limsup_{k\to\infty}\left(\frac1{2\epsilon_k} \| (w_k-u_k)_-\|_U^2 +\frac1{2\epsilon_k} \| (w_k+u_k)_-\|_U^2   \right)
         \\
        &\geq \Phi_0(u^*,w^*)+\frac12 \|u^*-\bar u\|^2_{U} + \frac12 \|w^*-\Bar{w}\|_{L^2(\Omega)}^2  \\
         &\quad
         + \limsup_{k\to\infty}\left(\frac1{2\epsilon_k} \| (w_k-u_k)_-\|_U^2 +\frac1{2\epsilon_k} \| (w_k+u_k)_-\|_U^2   \right).
     \end{split}
    \end{align}
    From these inequalities we obtain boundedness of the terms $ \frac1{2\epsilon_k} \| (w_k \pm u_k)_-\|_U^2$. Since $\epsilon_k \to 0$, it holds $ \| (w_k\pm u_k)_-\|_U \to 0$.
    Further, we have $w_k - u_k - (w_k-u_k)_- \in \{ v \in U: v\geq 0 \textrm{  a.e. on } I \times \Omega\}$.
    This set is weakly closed. Using $w_k - u_k \rightharpoonup w^*-u^*$ in $U$ and $(w_k-u_k)_- \to 0$ in $U$,\
    it follows $w^* - u^* \ge0$ almost everywhere on $I\times \Omega$.
    Similarly, one can show $w^* + u^*\geq 0$ almost everywhere on $I\times\Omega$, and therefore $(u^*,w^*)$ is admissible for problem \eqref{eq:minprob}-\eqref{eq:conduw}.
    By local optimality of $(\bar u, \Bar{w})$ we get together with the chain of inequalities \eqref{eq:chainInequLimsup}
    \[\Phi_0(\bar u,\Bar{w}) \ge \Phi_0(u^*,w^*)+\frac12 \|u^*-\bar u\|^2_{U} + \frac12 \|w^*-\Bar{w}\|_{L^2(\Omega)}^2 \geq \Phi_0(\bar u,\bar{w}),
    \]
    which implies  that $(\bar u,\Bar{w})=(u^*,w^*)$.
     The first inequality in \eqref{eq:chainInequLimsup} also gives $\Phi_0(\bar u,\Bar{w}) \ge \limsup_{k \to \infty} \Phi_{\epsilon_k}(u_k,w_k)$,
    which together with
Lemma \ref{lem_aux_strong_convergence} yields strong convergence $u_k \to \bar u$ in $U$ and $w_k \to \Bar{w}$ in $W$.
\end{proof}

Next, we consider a necessary optimality condition for problem \eqref{eq:auxprob}.
\begin{lemma}\label{tm:oc}
    For some $\epsilon>0$, let $(u_{\epsilon},w_{\epsilon})$ be a  solution of \eqref{eq:auxprob} with $\|u_{\epsilon}-\bar u\|_U < \rho$ and $\|w_{\epsilon}-\Bar{w}\|_W < \rho$.
    Then it holds
    \begin{equation}\begin{gathered}\label{eq:oc_aux}
        f'(u_{\epsilon})v+\alpha(u_{\epsilon}, v)_U + \beta (w_{\epsilon},z)_W+\gamma G_{\epsilon}'(w_{\epsilon})z + (u_{\epsilon}-\bar u, v)_U+(w_{\epsilon}-\Bar{w}, z)_{L^2(\Omega)} \\
        + \frac1{\epsilon} ( (w_{\epsilon}-u_{\epsilon})_-,\ z-v)_U + \frac1{\epsilon} (( w_{\epsilon}+u_{\epsilon})_-,\ z+v)_U = 0\ \forall (v,z)\in U\times W.
    \end{gathered} \end{equation}
\end{lemma}
Here, we used the notation
\[
 ( (w_{\epsilon} \pm u_{\epsilon})_-,\ z \pm v)_U
 := \int_{\Omega}\int_I (w_{\epsilon}(x) \pm u_{\epsilon}(t,x))_-  \cdot (z(x)\pm v(t,x)) \dd t \dd x.
\]
\begin{proof}
The map $z \mapsto (z_-)^2$ is continuously differentiable from $\R$ to $\R$.
Using dominated convergence, it is easy to verify that $v \mapsto \| v_-\|_{L^2(I\times\Omega )}^2$ is directionally
differentiable from $L^2(I\times\Omega)$ to $\R$.
In addition, $f$ and $G_{\epsilon}$ are Fréchet differentiable by Assumption \ref{ass:f} and Lemma \ref{tm:psieps}.
    Then the objective function of problem \eqref{eq:auxprob} is directionally differentiable, and a necessary optimality condition is
    \begin{multline*}
        f'(u_{\epsilon})(v-u_{\epsilon})+\alpha(u_{\epsilon}, v-u_{\epsilon})_U + \beta (w_{\epsilon},z-w_{\epsilon})_W+\gamma G_{\epsilon}'(w_{\epsilon})(z-w_{\epsilon}) \\
        + (u_{\epsilon}-\bar u, v-u_{\epsilon})_U+(w_{\epsilon}-\Bar{w}, z-w_{\epsilon})_{L^2(\Omega)}
        + \frac1{\epsilon} ( (w_{\epsilon}-u_{\epsilon})_-,\ z-w_{\epsilon}-(v-u_{\epsilon}))_U \\
        + \frac1{\epsilon} ( (w_{\epsilon}+u_{\epsilon})_-,\ z-w_{\epsilon}+v-u_{\epsilon})_U \geq 0
    \end{multline*}   for all  $(v,z)\in U\times W$ with $\|v-\bar u\|_U \leq \rho$ and $\|z-\Bar{w}\|_W \leq \rho$. Using the fact that $\|u_{\epsilon}-\bar u\|_U < \rho$ and $\|w_{\epsilon}-\Bar{w}\|_W < \rho$ one obtains the optimality condition \eqref{eq:oc_aux}.
\end{proof}
Note that for solutions $(u_k,w_k)$ corresponding to $\epsilon_k\searrow 0$, the condition $\|u_k-\bar u\|_U < \rho$ and $\|w_k-\Bar{w}\|_W < \rho$ is always satisfied for $k$ large enough as a consequence of Theorem \ref{tm:conv_uw}. \\ The next theorem shows that solutions $(u_{\epsilon}, w_{\epsilon})$ to the penalized and unconstrained optimization problem still satisfy a necessary condition for admissibility.

\begin{lemma}\label{tm:wepsgeq0}
    Let $(u_{\epsilon},w_{\epsilon})$ be a solution to problem \eqref{eq:auxprob}. Then it holds $w_{\epsilon}\geq 0$.
\end{lemma}
\begin{proof}
Define $w_{\epsilon}^- := (w_{\epsilon})_-$.
Then we get the following inequalities
\[
 G_{\epsilon}'(w_{\epsilon})w_{\epsilon}^- \ge0, \quad
 (w_{\epsilon} , w_{\epsilon}^-)_{L^2(\Omega)} \ge 0, \quad
 ( - \Bar{w}, w_{\epsilon}^-)_{L^2(\Omega)}\ge 0,
\]
 where we used the expression of $G_{\epsilon}'$ from Lemma \ref{tm:psieps} and $\bar w \ge0$.
Furthermore, by Lemma \ref{lm:inequPosPart} we have
\[
 (w_{\epsilon},w_{\epsilon}^-)_W  = (w_{\epsilon},w_{\epsilon} )_W  - (w_{\epsilon},w_{\epsilon}^+ )_W \ge0.
\]
In addition, we get $((w_{\epsilon} \pm u_{\epsilon})_-, w_{\epsilon}^-)_U  \ge 0$.
Now, we test \eqref{eq:oc_aux} with $(0,w_{\epsilon}^-)$ to obtain
\begin{multline*}
    \beta (w_{\epsilon}, w_{\epsilon}^-)_W + \gamma G_{\epsilon}'(w_{\epsilon})w_{\epsilon}^- + (w_{\epsilon}, w_{\epsilon}^-)_{L^2(\Omega)} -(\Bar{w}, w_{\epsilon}^-)_{L^2(\Omega)}
    \\
    + \frac1{\epsilon}( (w_{\epsilon}-u_{\epsilon})_- + (w_{\epsilon}+u_{\epsilon})_- , w_{\epsilon}^-)_U = 0.
\end{multline*}
Since by the above considerations every summand in this equation is non-negative,
it follows that all summands are zero.
In particular, this implies
$ 0 = (w_{\epsilon}, w_{\epsilon}^-)_{L^2(\Omega)} = (w_{\epsilon}^-, w_{\epsilon}^-)_{L^2(\Omega)}$,
so that $w_{\epsilon}^-= 0$ and $w_{\epsilon} \ge 0$ almost everywhere on $\Omega$.
\end{proof}

\begin{lemma} \label{tm:mult_bounded}
    Let $(u_k,w_k)$ be the sequence of solutions to the smoothed problem \eqref{eq:auxprob} corresponding to a sequence $(\epsilon_k) \searrow 0$ and let $\lambda_k := G_{\epsilon_k}'(w_k)$, $\mu_k^1:= \frac1{\epsilon_k} (w_k-u_k)_-$ and
    $\mu_k^2:= \frac1{\epsilon_k}(w_k +u_k)_-$.
    Then there are $\Bar{\mu}^1$ and $\Bar{\mu}^2$ in $U$
    such that for a subsequence $(k_n)$ it holds $\mu_{k_n}^1 \rightharpoonup \bar{\mu}^1$ and $\mu_{k_n}^2 \rightharpoonup \bar{\mu}^2$ in $U$. Furthermore, it holds $\lambda_{k_n} \to \Bar{\lambda}  $ in $W^*$ for some $\Bar{\lambda}$ in $W^*$.
\end{lemma}
\begin{proof}
   By testing equation \eqref{eq:oc_aux} with $v := - (w_k - u_k)_- $ and $z=0$, we get \begin{multline*}
       - f'(u_k) (w_k - u_k)_-+ \alpha ( u_k, -(w_k - u_k)_-)_U + (u_k-\bar u, -(w_k - u_k)_-)_U \\+ \frac1{\epsilon_k} \| (w_k - u_k)_- \|_U^2  + \frac1{\epsilon_k} ( (w_k + u_k)_- ,  - (w_k - u_k)_- )_U=0.
    \end{multline*} With Lemma \ref{tm:wepsgeq0} the last term on the left-hand side vanishes.

    Rearranging and using Cauchy-Schwarz inequality leads to
    \[
        \frac1{\epsilon_k} \| (w_k - u_k)_- \|_U^2   \leq \left( \|f'(u_k)\|_{U} + \alpha\|u_k\|_U + \|u_k-\bar u\|_U  \right)  \| (w_k - u_k)_- \|_U ,
    \]
    which implies
    \[
    \frac1{\epsilon_k} \| (w_k - u_k)_- \|_U\leq  \left( \|f'(u_k)\|_{U} + \alpha\|u_k\|_U + \|u_k-\bar u\|_U  \right).
    \]

    Since $(f'(u_k))$ and $(u_k)$ are bounded sequences in $U$, we obtain  boundedness of $\frac1{\epsilon_k} \| (w_k - u_k)_- \|_U$.

    Analogously, one obtains boundedness of $\frac1{\epsilon_k} \| (w_k + u_k)_- \|_U$. Therefore, weak convergence $\mu_k^1 \rightharpoonup \bar{\mu}^1$ and $\mu_k^2 \rightharpoonup \bar{\mu}^2$ in $U$ along a subsequence $(k_n)$ for some weak limit points $\bar{\mu}^1,\ \bar{\mu}^2 \in U$ follows.

    Using the optimality condition \eqref{eq:oc_aux} with $v\equiv 0$, we get
    \[
      \gamma \braket{\lambda_k,z } = -\beta (w_k,z)_W - (w_k-\Bar{w}, z)_{L^2(\Omega)} - (\mu_k^1,z)_U -  (\mu_k^2, z)_U\quad \forall z\in W,
    \]
    where $(\mu_k^i,z)_U = \int_\Omega \int_I \mu_k^i(t,x) \dd t\ z(x) \dd x$ for $i=1,2$.
    Due to the compact embedding of $U$ into $W ^*$, all terms on the right-hand side converge in $W^*$ along the subsequence $(k_n)$,
    and we obtain $\lambda_{k_n} \to \Bar{\lambda}$ in $W^*$.
\end{proof}
The previous Lemma \ref{tm:mult_bounded} explains the use of the penalization of the constraint $(u,w)\in K$ in the auxiliary problem. By first showing boundedness and weak convergence of the sequences $(\mu_k^1), (\mu_k^2)$, one can quite directly deduce boundedness and convergence of $(\lambda_k)$. Without the penalization, it would be necessary to somehow separate $(\lambda_k)$ and the multipliers corresponding to the constraint. 

Concerning the properties of the limit $\bar \lambda$, we have the following result.

\begin{lemma}
\label{tm:theory_lambda_conv}
    Let $\epsilon_k\searrow 0$, and let $(u_k,w_k)$ be a sequence of solutions to \eqref{eq:auxprob} and
    $\lambda_k$ the sequence of multipliers as defined above. Suppose $w_k\to \Bar{w}$ in $W$ and $\lambda_k \to \Bar{\lambda}$ in $W^*$. Then \begin{align*}
        \braket{\lambda_k,w_k}_W \to \braket{\Bar{\lambda},\Bar{w}}_W = p \int_{\Omega}|\Bar{w}|^p \dd x.
    \end{align*}
\end{lemma}
\begin{proof}
	The proof of \cite[Lemma 5.5]{sparse} can be transferred to this setting here,
	as it only uses the definition of $\lambda_k$ via the smoothing scheme and the convergence $w_k\to \Bar{w}$ in $W \hookrightarrow L^2(\Omega)$.
\end{proof}

Combining the results above, we obtain an optimality condition for the original problem \eqref{eq:minprob}.
Here, we will obtain Lagrange multipliers to both constraints encoded in the set $K$, i.e., to the constraints $w - u \ge0$ and $w+u \ge0$.
Define
\begin{equation} \label{eq_def_C}
 C:= \{ v \in U: \ v \ge 0 \}, \quad I_C(v):=\begin{cases} 0& \text{ if } v\in C, \\ +\infty &\text{ otherwise.}\end{cases}
\end{equation}
Then $\mu \in U=U^*$ belongs to the convex subdifferential of $I_C$  at $\bar v\in C$ if and only if $( \mu, v - \bar v)_U \le 0$ for all $v\in C$, see \cite[Example 2.31]{convex_opt}.

\begin{theorem}
\label{thm_nec_opt_cond}
    Let $(\bar u,\Bar{w})$ be a local solution of the original problem \eqref{eq:minprob}. Then there are $\Bar{\lambda}\in W^*$ and $\Bar{\mu}^1, \Bar{\mu}^2 \in U $
    such that \begin{align}\label{eq:oc_lim}
        f'(\bar u)v+\alpha(\bar u,v)_U+\beta(\Bar{w},z)_W+
        \gamma \braket{\Bar{\lambda},z}_{W} + (\bar{\mu}^1, z-v)_U + (\bar{\mu}^2, z+v)_U = 0
    \end{align}
   for all $ (v,z)\in U\times W$ and it holds\begin{align}\label{eq:oc_lambda}
     \braket{\Bar{\lambda},\Bar{w}}_W = p\int_{\Omega} |\Bar{w}|^p \dd x.
    \end{align}
    Moreover, $\bar{\mu}^1 \in \partial I_C(\bar w - \bar u) $ and $\bar{\mu}^2 \in \partial I_C(\bar w + \bar u)$.
\end{theorem}

Here, we used the notation
\[
  (\bar{\mu}^1, z-v)_U  = \int_\Omega \int_I \bar\mu^1(t,x)( z(x) - v(t,x)) \dd t \dd x.
\]
for $v\in U$, $z\in W$.

\begin{proof}
    Let $(u_k,w_k)$ be the sequence of solutions to the smoothed problem \eqref{eq:auxprob} corresponding to a sequence $(\epsilon_k) \searrow 0$. Equation \eqref{eq:oc_lim} is a consequence of passing to the limit in \eqref{eq:oc_aux} along a subsequence using Lemma \ref{tm:mult_bounded}.
    Property \eqref{eq:oc_lambda} follows from Lemma \ref{tm:theory_lambda_conv}.
    Next, we show
     $\bar{\mu}^1 \in \partial I_C(\bar w - \bar u) $.
     Let $v\in C$. Using the definition of
     $\mu_k^1$ as $\mu_k^1=\frac1{\epsilon_k} (w_k-u_k)_- \le 0$ in Lemma \ref{tm:mult_bounded},
     we get
    \begin{align*}
        (\mu_k^1,v-(w_k -u_k) )_U = (\mu_k^1, v )_U - \frac1{\epsilon_k} \|(w_k-u_k)_-\|_U^2 \leq 0.
    \end{align*}
    Convergence of $w_k \to \bar w$ in $W$, $u_k\to \bar u$ in $U$, and $\mu_k^1 \rightharpoonup \bar\mu^1$ in $U$ implies
    \begin{align*}
        (\bar{\mu}^1, v-(\Bar{w}-\bar u) )_U =  \lim_{k \to \infty}  (\mu_k^1,v-(w_k - u_k) )_U \leq 0.
    \end{align*}
    The claim follows now by the definition of the convex subdifferential. The corresponding statement for $\bar{\mu}^2$ can be shown analogously.
\end{proof}
The statements
    $\bar{\mu}^1 \in \partial I_C(\bar w - \bar u) $ and $\bar{\mu}^2 \in \partial I_C(\bar w + \bar u)$
    are equivalent to the complementarity system
    \begin{align*}
&\Bar{\mu}^1 \leq 0, \hspace{3mm} \Bar{w}-\bar u \geq 0, \hspace{3mm} (\Bar{\mu}^1, \Bar{w}-\bar u)_U=0 \\ \text{and} \hspace{3mm}
    &\Bar{\mu}^2 \leq 0,  \hspace{3mm}\Bar{w}+\bar u \geq 0,  \hspace{3mm}(\Bar{\mu}^2, \Bar{w}+\bar u)_U=0.
\end{align*}

\subsection{Regularity results}

Now, we want to investigate some regularity results for solutions $(\bar u, \bar w)$.
The following lemma is used to prove $L^{\infty}$-regularity of the solutions.

\begin{lemma}
{\normalfont (  \cite[Lemma 3.4]{49snObstacle})}
\label{tm:aux_linf}
    Let    $r>1$
    and $q>2$ be such that $\frac1{r}+\frac{2}{q}<1$ and  $W \hookrightarrow L^q(\Omega)$.
    Then there is $c>0$ such that if $g\in L^r(\Omega)$, $w\in W$, and
    \begin{align*}
        (w,w-w_n)_W \leq \int_{\Omega} g (w-w_n) \dd x \hspace{5mm} \forall n \in \N
    \end{align*} for $w_n:=  \max(-n,\min(w,n))$ it follows \begin{align*}
        \|w\|_{L^{\infty}(\Omega)} \leq c\|g\|_{L^r(\Omega)}.
    \end{align*}
\end{lemma}

\begin{theorem}\label{tm:linf_reg}
    Let  $(\bar u, \bar w)$ be a local solution of problem \eqref{eq:minprob}
    such that the map $x\mapsto \|f'(\bar u)(\cdot,x)\|_{L^1(I)}$ is in $L^r(\Omega)$ with $r$ as in  Lemma \ref{tm:aux_linf} above. Then it holds $\Bar{w}\in L^{\infty}(\Omega)$ and   $\bar u\in L^{\infty}(I\times \Omega)$.
\end{theorem}
Note that we identify $f'(\bar u)$ with its Riesz representative $f'(\bar u) \in U = L^2(I\times \Omega)$.
\begin{proof}
    Using $\bar w\ge 0$ and Lemma \ref{tm:propFrac} (ii), let \begin{equation*}
        \Bar{w}_{n} := \max(-n,\min(\Bar{w},n))=\min(\Bar{w},n) \in W, \hspace{5mm} \bar u_{n} := \max(-n,\min(\bar u,n)) \in U.
    \end{equation*}
    Let $(u_k,w_k)$ be the sequence of solutions to the smoothed problem \eqref{eq:auxprob} corresponding to a sequence $(\epsilon_k) \searrow 0$.
Testing \eqref{eq:oc_aux} with $(\bar u-\bar u_n,\Bar{w}-\bar{w}_n)$ yields
\begin{align}\begin{split}
\label{eq:forlinf}
    f'(u_k)(\bar u-\bar u_n)&+\alpha (u_k,\bar u-\bar u_n)_U+\beta(w_k,\Bar{w}-\bar{w}_n)_W+\gamma G_{\epsilon_k}'(w_k)(\Bar{w}-\bar{w}_n) \\&+(u_k-\bar u,\bar u-\bar u_n)_U+(w_k-\Bar{w},\Bar{w}-\bar{w}_n)_{L^2(\Omega)}
    \\
    &+ ( \mu^1_k , \Bar{w}-\bar{w}_n - (\bar u-\bar u_n))_U
    +( \mu^2_k , \Bar{w}-\bar{w}_n + \bar u-\bar u_n)_U=0. 	\end{split}
    \end{align}
    Since $w_k\ge 0$ by Lemma \ref{tm:wepsgeq0}, one can observe
\begin{align*} \gamma G_{\epsilon_k}'(w_k)(\Bar{w}-\bar{w}_n)\geq 0,
\end{align*}
so when passing to the limit $k\to\infty$ in \eqref{eq:forlinf}
along a subsequence as in Lemma \ref{tm:mult_bounded},
we obtain
\begin{multline} \label{eq_15}
  \alpha (\bar u,\bar u-\bar u_n)_U+\beta(\Bar{w},\Bar{w}-\bar{w}_n)_W + (\Bar{\mu}^1, \Bar{w}-\bar{w}_n - (\bar u-\bar u_n))_U  +(\Bar{\mu}^2, \Bar{w}-\bar{w}_n + \bar u-\bar u_n)_U \\
  \leq   f'(\bar u)(\bar u_n-\bar u).
  \end{multline}
By construction, it follows $\alpha (\bar u,\bar u-\bar u_n)_U \geq 0$.
In addition, the truncated functions satisfy $|\bar u_n| \le \bar{w}_n$.
Using $\bar{\mu}^1 \in \partial I_C(\bar w - \bar u) $ and $\bar{\mu}^2 \in \partial I_C(\bar w + \bar u)$
with $C$ as in \eqref{eq_def_C} it follows
\[
 (\Bar{\mu}^1, \Bar{w}-\bar{w}_n - (\bar u-\bar u_n))_U\geq 0, \qquad  (\Bar{\mu}^2, \Bar{w}-\bar{w}_n + \bar u-\bar u_n)_U \geq 0.
\]
Inequality \eqref{eq_15}
can thus be reduced to \begin{equation} \label{eq:aux_linf_uwineq}
    \beta(\Bar{w}, \Bar{w}-\Bar{w}_n)_W \leq  f'(\bar u)(\bar u_n-\bar u).
\end{equation}
Here, we want to replace $\bar u_n-\bar u$ by $\Bar{w}-\Bar{w}_n$ to be able to apply Lemma \ref{tm:aux_linf}.
To this end, we want to prove
\begin{equation} \label{eq_18}
    |\bar{u}(t,x)-\bar u_n(t,x)|\leq \Bar{w}(x) - \Bar{w}_n(x) \quad \textrm{for a.a. } (t,x) \in I\times\Omega.
\end{equation}
If $|\bar u(t,x)|\le n$, then $ |\bar{u}(t,x)-\bar u_n(t,x)| =0\leq \Bar{w}(x) - \Bar{w}_n(x)$.
Suppose $\bar u(x)>n$. Due to the feasibility of $(\bar u,\Bar{w})$
it follows $\Bar{w}(x)\ge\bar{u}(x)>n$ and
 \[  |\bar{u}(t,x)-\bar u_n(t,x)| = \bar{u}(t,x) -n \leq \Bar{w}(x) - n = \Bar{w}(x) - \Bar{w}_n(x).\]
The case  $\bar u(t,x)<-n$ can be treated similarly.
This proves \eqref{eq_18}.

By Hölder's inequality and \eqref{eq_18}, we get
\begin{align*}
    \int_{\Omega}\int_I f'(\bar u) (\bar u-\bar u_n) \dd t \dd x &\leq \int_{\Omega}\int_I |f'(\bar u)| |\bar u-\bar u_n| \dd t \dd x \\
    &\leq \int_{\Omega} (\Bar{w}(x)-\Bar{w}_n(x))\int_I |f'(\bar u)(t,x)|  \dd t \dd x \\
    &= \int_{\Omega} (\Bar{w}(x)-\Bar{w}_n(x))  \|f'(\bar u)(\cdot,x)\|_{L^1(I)}  \dd x.
\end{align*}
Using this estimate in \eqref{eq:aux_linf_uwineq} yields
\begin{align*}
    \beta(\bar{w},\Bar{w}-\Bar{w}_n)_W \leq  \int_{\Omega} \|f'(\bar u(\cdot,x))\|_{L^1(I)} (\Bar{w}-\Bar{w}_n) \dd x \hspace{5mm} \forall n \in \N.
\end{align*}
By assumption on $f'$, one can apply Lemma \ref{tm:aux_linf} with $g(x) = \|f'(\bar u)(\cdot,x)\|_{L^1(I)} $ to obtain $\bar w\in L^\infty(\Omega)$.
The claim $\bar u\in L^\infty(I \times \Omega)$ follows from $|\bar u|\le \bar w$.
\end{proof}

For some choices of $d$ and $s$, the assumption on $f'$ in the previous theorem is always satisfied.
\begin{corollary}
    Let $d$ and $s$ such that $d \leq  2s$ or $d>2s$ and $\frac{2d}{d-2s} >4$. Then it holds $\bar{w}\in L^{\infty}(\Omega)$ and   $\bar u\in L^{\infty}(I\times \Omega)$.
\end{corollary}
\begin{proof}
	Since
    \[
    \int_{\Omega}\left(\int_I |f'(\bar u)(t,x)| \dd t\right)^2\dd x \le T  \int_{\Omega}\int_I |f'(\bar u)(t,x)|^2 \dd t \dd x < \infty,
     \]
    the function $x\mapsto \| f'(\bar u)(\cdot,x)\|_{L^1(I)}$ is in $L^2(\Omega)$. By assumption on $d$ and $s$ and the fractional Sobolev embeddings from Lemma \ref{tm:propFrac} (iv), there is $q>4$ such that $W$ embeds into $L^q(\Omega)$. Therefore, $\frac{1}{2}+\frac{2}{q}<1$ holds, and we can apply Theorem \ref{tm:linf_reg}.
\end{proof}

\begin{theorem}  Let $\epsilon_k\searrow 0$, $(u_k,w_k)$ the sequence of solutions to \eqref{eq:auxprob} with multipliers $\lambda_k\to \Bar{\lambda}$ and $\mu^1_k, \mu^2_k$  as defined previously.
Then $( \lambda_{k})$ is bounded in $L^1(\Omega)$, and $\Bar{\lambda}$ can be identified with a measure from $\mathcal{M}(\Omega)=C_0(\Omega)^*$.
\end{theorem}
\begin{proof}
    The proof works as the proof of \cite[Lemma 5.9]{sparse}.
    We test the optimality condition \eqref{eq:oc_aux} with $(0,s_{k,n})$, where $s_{k,n} := \min(n \cdot w_k,1) \approx \operatorname{sgn}(w_k)$. This yields
    \begin{multline*}
    \beta (w_k,s_{k,n})_W+\gamma \braket{ \lambda_k ,s_{k,n}}_W +(w_k-\Bar{w}, s_{k,n})_{L^2(\Omega)}
     +  ( \mu_k^1, s_{k,n})_U +  ( \mu_k^2, s_{k,n})_U = 0.
    \end{multline*}
    With $(w_k,s_{k,n})_{L^2(\Omega)} \geq 0$ and $(w_k,s_{k,n})_W\geq \frac1{n} \|s_{k,n}\|_W^2$ from Lemma \ref{tm:propFrac},
    we obtain
    \begin{equation*}\begin{aligned}
        \beta \frac1{n}\|s_{k,n}\|_W^2 + \gamma \braket{ \lambda_k ,s_{k,n}}_W & \leq (\Bar{w}-w_k,s_{k,n})_{L^2(\Omega)} -(\mu_k^1, s_{k,n})_U - ( \mu_k^2, s_{k,n})_U \\
        &\leq \|\Bar{w}\|_{L^1(\Omega)} + \|\mu^1_k\|_{L^1(I \times\Omega)} + \|\mu^2_k\|_{L^1( I \times \Omega)}
        \end{aligned}
    \end{equation*}
    By dominated convergence, we can pass to the limit $n\to\infty$ in
    \begin{align*}
       \braket{ \lambda_k ,s_{k,n}}_W = \int_{\Omega} 2 w_k \psi_{\epsilon_k}'(w_k^2) s_{k,n} \dd x \to \int_{\Omega} 2 |w_k| \psi_{\epsilon_k}'(w_k^2)  \dd x = \| \lambda_k\|_{L^1(\Omega)}.
    \end{align*} This proves the boundedness of $( \lambda_{k})$  in $L^1(\Omega)$.
We identify $L^1(\Omega)$ with a subspace of $C_0(\Omega)^*$ using the canonical embedding,
which associates each $\lambda_k \in L^1(\Omega)$ with the functional $\phi \mapsto \int_\Omega \lambda_k \phi \dd x$.
Since $C_0(\Omega)$ is separable, we have (after possibly extracting a subsequence)
$\lambda_k \rightharpoonup^* \Tilde{\lambda}$ in $C_0(\Omega)^*$.
The density result in Lemma \ref{tm:propFrac} yields $\Tilde{\lambda}=\Bar{\lambda}$.
\end{proof}

\section{Iterative scheme}\label{sec:algo}

In this section, we will employ a majorize-minimization scheme to solve \eqref{eq:minprob}--\eqref{eq:conduw}.
The resulting method is Algorithm \ref{alg:alg_nonpen}.
For the well-posedness and convergence analysis of the algorithm we assume that $f$ satisfies the following properties in addition to Assumption \ref{ass:f}.
\begin{assumption}\label{ass:f'}
    The function $f'$ is completely continuous, i.e. $u_k \rightharpoonup \bar u$ implies $f'(u_k) \to f'(\bar u)$ in $U$.  Furthermore, $f'$ is Lipschitz continuous on bounded sets.
\end{assumption}
Lipschitz continuity of $f'$ is used in Lemma \ref{tm:algoWelldef} to show that condition \eqref{eq:alg_nonpen_fcond} in the algorithm can be satisfied.
The complete continuity of $f'$ is necessary to  pass to the limit in the optimality condition of \eqref{eq:minprob_algo_nonpen}
along weakly converging sequences of iterates in Lemma \ref{tm:conv_uk}.

\begin{algorithm}
\caption{Iterative scheme}\label{alg:alg_nonpen}
\begin{algorithmic}[1]
\State Choose $b>1$, $L>0$ and a monotonically decreasing sequence $(\epsilon_k)$ with $\epsilon_k \searrow 0$. Set $k=0$.
\State Find the smallest number $L_k$ from $\{{L}b^l: l\geq 0\}$ such that the solution $(u_{k+1},w_{k+1})$ of
\begin{multline} \label{eq:minprob_algo_nonpen}
    \underset{ (u,w)\in K }{\min} \varphi_k(u,w) := f'(u_k)(u-u_k) + \frac{\alpha}{2}\|u\|_U^2 + \frac{\beta}{2}\|w\|^2_W
    \\
    +\gamma \int_{\Omega} \psi'_{\epsilon_k}(w_k^2)(w^2-w_k^2) \dd x + \frac{L_k}{2}\|u-u_k\|_U^2
\end{multline}
satisfies \begin{align} \label{eq:alg_nonpen_fcond}
    f(u_{k+1}) \leq f(u_k) + f'(u_k)(u_{k+1}-u_k) + L_k\|u_{k+1}-u_k\|_U^2.
\end{align}
\State Set $k:= k+1$ and go to step 2.
\end{algorithmic}
\end{algorithm}

Algorithm \ref{alg:alg_nonpen} is similar to the one introduced in \cite{sparse}.
The term $f'(u_k)(u-u_k)+ \frac{L_k}{2}\|u-u_k\|_U^2$ is a quadratic approximation of $f(u)-f(u_k)$, which is inspired from proximal gradient methods.
The term $\psi'_{\epsilon_k}(w_k^2)(w^2-w_k^2)$ is an upper bound of $\psi_{\epsilon_k}(w^2)- \psi_{\epsilon_k}(w_k^2)$ due to the concavity of $\psi_{\epsilon_k}$, see also \cite{lp_cont}.
Thus the cost function \eqref{eq:minprob_algo_nonpen} is quadratic in $(u,w)$.
The condition \eqref{eq:alg_nonpen_fcond} ensures that the sequence of function values $(\Phi_{\epsilon_k}(u_k,w_k))$ is monotonically decreasing.
In difference to \cite{sparse},
we have to solve a constrained optimization problem in each step instead of an unconstrained one,
which complicates the convergence analysis considerably.

Subproblem \eqref{eq:minprob_algo_nonpen} is coercive and strictly convex, and therefore admits a unique solution for every $(u_k,w_k)$ and $L_k$.
 Its solution $(u_{k+1},w_{k+1})$ satisfies the optimality condition
 \begin{equation} \begin{gathered}\label{eq:oc_nonpen_alg_ineq}
     f'(u_k)(v-u_{k+1}) + \alpha (u_{k+1},v-u_{k+1})_U + \beta(w_{k+1},z-w_{k+1})_W \\+ 2 \gamma \int_{\Omega}\psi'_{\epsilon_k}(w_k^2)w_{k+1}(z-w_{k+1}) \dd x + L_k(u_{k+1}-u_k,v-u_{k+1})_U\geq 0 \hspace{5mm} \forall (v,z)\in K. \end{gathered}
 \end{equation}

The next lemma helps to prove well-definedness of Algorithm \ref{alg:alg_nonpen} in the subsequent lemma.

\begin{lemma}
\label{tm:descent_lemma}
    Let $M>0$, $\rho>0$, and $u\in U$ such that ${\|f'(u)-f'(v)\|_{U}\leq M \|u-v\|_U}$ for all $v\in B_{\rho}(u)$. Then \begin{align*}
        f(v)\leq f(u) + f'(u)(v-u)+\frac{M}{2}\|v-u\|_U^2 \hspace{5mm} \forall v \in B_{\rho}(u).
    \end{align*}
\end{lemma}
\begin{proof}
    This follows from the mean-value theorem \cite[Prop. 3.3.4]{9nonsmoothAna}.
\end{proof}
\begin{lemma}\label{tm:algoWelldef}
For each iterate $(u_k,w_k)$ of Algorithm \ref{alg:alg_nonpen}, there is  $L_k\geq 0$ and a solution $(u_{k+1},w_{k+1})$ of \eqref{eq:minprob_algo_nonpen} such that condition \eqref{eq:alg_nonpen_fcond} is satisfied.
\end{lemma}
\begin{proof}
    The proof can be done as in \cite[Lemma 7.2]{sparse}. Let $(v_{n,k},z_{n,k})$ be the solution of \begin{multline*}
        \underset{(u,w) \in K}{\min} f(u_k) + f'(u_k)(u-u_k) + \frac{\alpha}{2}\|u\|_U^2 + \frac{\beta}{2}\|w\|^2_W \\+ \gamma \int_{\Omega}\psi_{\epsilon_k}(w_k^2)+\psi'_{\epsilon_k}(w_k^2)(w^2-w_k^2) \dd x + \frac{n}{2}\|u-u_k\|_U^2
    \end{multline*} for $n\in \N$.
    By definition of $(v_{n,k},z_{n,k})$ it holds \begin{multline*}
        f(u_k) + \frac{\alpha}{2}\|u_k\|_U^2 + \frac{\beta}{2}\|w_k\|^2_W+ \gamma \int_{\Omega}\psi_{\epsilon_k}(w_k^2) \\ \geq f(u_k) + f'(u_k)(v_{n,k}-u_k) + \frac{\alpha}{2}\|v_{n,k}\|_U^2 + \frac{\beta}{2}\|z_{n,k}\|^2_W \\+ \gamma \int_{\Omega}\psi_{\epsilon_k}(w_k^2)+\psi'_{\epsilon_k}(w_k^2)(z_{n,k}^2-w_k^2) \dd x + \frac{n}{2}\|v_{n,k}-u_k\|_U^2.
    \end{multline*}
    Concavity and non-negativity of $\psi_{\epsilon}$ yield
    \[
    \psi_{\epsilon_k}(w_k^2)+\psi'_{\epsilon_k}(w_k^2)(z_{n,k}^2-w_k^2) \geq \psi_{\epsilon_k}(z_{n,k}^2) \geq 0,
    \]
    and we obtain \begin{align}\label{eq:ineq_for_lk_bounded}\begin{split}
        \frac{\alpha}{2}\|u_k\|_U^2 + \frac{\beta}{2}\|w_k\|^2_W+ \gamma \int_{\Omega}\psi_{\epsilon_k}(w_k^2) &\geq  f'(u_k)(v_{n,k}-u_k) + \frac{n}{2}\|v_{n,k}-u_k\|_U^2 \\ &\geq \frac{n}{4}\|v_{n,k}-u_k\|_U^2 - \frac1{n} \|f'(u_k)\|_{U}^2, \end{split}
    \end{align}
	using a scaled version of Young's inequality.
	This shows $v_{n,k}\to u_k$ in $U$ for $n\to \infty$.
    By local Lipschitz continuity of $f'$ from Assumption \ref{ass:f'}, there are $M>0$ and $N>0$ such that for all $n>N$ we have
    \begin{align*}
        f(v_{n,k})\leq f(u_k) + f'(u_k)(v_{n,k}-u_k) + \frac{M}{2}\|v_{n,k}-u_k\|_U^2
    \end{align*} by Lemma \ref{tm:descent_lemma}.
    Hence, condition \eqref{eq:alg_nonpen_fcond} is satisfied for $L_k \ge \max(N,\frac M2)$.
\end{proof}

\begin{lemma}\label{tm:lk_bounded_nonpenalgo}
	If the sequence $(u_k,w_k)$ of iterates of Algorithm \ref{alg:alg_nonpen} is bounded in $U\times W$, then the sequence $(L_k)$ is bounded.
\end{lemma}
\begin{proof}
 Suppose $(u_k)$ and $(w_k)$ are bounded with $\max(\|u_k\|_U, \|w_k\|_W)\leq R_1$. By \eqref{eq:ineq_for_lk_bounded} there is $R_2>0$ such that $\|v_{n,k}-u_k\|_U^2 \leq R_2/n$,
and therefore $\|v_{n,k}\|_U \leq R_1 + 1$ for all $n>R_2$ and all $k$. Then condition \eqref{eq:alg_nonpen_fcond} is satisfied for $L_k\geq \max(R_2,\frac{M}{2})$
with $M$ being the Lipschitz modulus of $f'$ on $B_{R_1+1}(0)$. Thus, $L_k \leq \max(R_2,\frac{M}{2})b$ by the choice of $L_k$ in Algorithm \ref{alg:alg_nonpen}.
\end{proof}

\begin{lemma}\label{lm:mon_phi_nonpen}
For a sequence $(u_k,w_k,L_k)$ generated by Algorithm \ref{alg:alg_nonpen}
it holds  \begin{multline*}
    \Phi_{\epsilon_{k+1}}(u_{k+1},w_{k+1})+\frac{\alpha}{2}\|u_{k+1}-u_k\|_U^2+\frac{\beta}{2}\|w_{k+1}-w_k\|_W^2+\gamma\int_{\Omega}\psi'_{\epsilon_k}(w_k^2)(w_{k+1}-w_k)^2\dd x \\ \leq \Phi_{\epsilon_k}(u_k,w_k),
\end{multline*}
where $\Phi_\epsilon$ was defined in \eqref{eq:phi_def}.

The sequences $(u_k)$, $(w_k)$, and $(f'(u_k))$ are bounded in $U$, $W$ and $U$, respectively.
In addition, $\|u_{k+1}-u_k\|_U \to 0$ and $\|w_{k+1}-w_k\|_W \to 0$ for $k \to\infty$.
\end{lemma}
\begin{proof}
The proof works mainly as the one of \cite[Lemma 7.3]{sparse}.
    We test the optimality condition \eqref{eq:oc_nonpen_alg_ineq} with $(u_k,w_k)\in K$ to obtain \begin{multline*}
        f'(u_k)(u_k-u_{k+1}) + \alpha (u_{k+1},u_k-u_{k+1})_U + \beta(w_{k+1},w_k-w_{k+1})_W \\+ 2 \gamma \int_{\Omega}\psi'_{\epsilon_k}(w_k^2)w_{k+1}(w_k-w_{k+1}) \dd x + L_k(u_{k+1}-u_k,u_k-u_{k+1})_U\geq 0.
    \end{multline*}
    Applying the formula $a(a-b)=\frac12((a-b)^2+a^2-b^2)$ yields
    \begin{multline*}
        f'(u_k)(u_k-u_{k+1}) - \frac{\alpha}{2} ( \|u_{k+1}\|_U^2 + \|u_{k+1}-u_k\|_U^2 -\|u_k\|_U^2) \\
        - \frac{\beta}{2}( \|w_{k+1}\|_W^2 + \|w_{k+1}-w_k\|_W^2-\|w_k\|_W^2)\\ -  \gamma \int_{\Omega}\psi'_{\epsilon_k}(w_k^2)(w_{k+1}^2 + (w_{k+1}-w_k)^2 -w_k^2) \dd x - L_k\|u_{k+1}-u_k\|_U^2 \geq 0.
    \end{multline*}
    Adding inequality \eqref{eq:alg_nonpen_fcond},
    rearranging, using concavity of $t\mapsto \psi_{\epsilon_k}(t)$ and monotonicity of the sequence $(\epsilon_k)$ leads to
    \begin{align*}
        \Phi_{\epsilon_k}(u_k,w_k) & = f(u_k)+ \frac{\alpha}{2} \|u_k\|_U^2 + \frac{\beta}{2} \|w_k\|_W^2  +  \gamma \int_{\Omega}\psi_{\epsilon_k}(w_k^2) \dd x \\
        &\geq f(u_{k+1}) + \frac{\alpha}{2} \|u_{k+1}\|_U^2 + \frac{\beta}{2} \|w_{k+1}\|_W^2 + \gamma \int_{\Omega} \psi_{\epsilon_k}(w_{k+1}^2) + \frac{\alpha}{2} \|u_{k+1}-u_k\|_U^2 \\
        &\quad + \frac{\beta}{2}  \|w_{k+1}-w_k\|_W^2  + \gamma \int_{\Omega} \psi'_{\epsilon_k}(w_k^2) (w_{k+1}-w_k)^2 \\
        &\geq \Phi_{\epsilon_{k+1}}(u_{k+1},w_{k+1}) + \frac{\alpha}{2} \|u_{k+1}-u_k\|_U^2 + \frac{\beta}{2} \|w_{k+1}-w_k\|_W^2 \\ &+ \gamma \int_{\Omega} \psi'_{\epsilon_k}(w_k^2) (w_{k+1}-w_k)^2 ,
    \end{align*}
    which is the first claim. 
    Then the sequence $( \Phi_{\epsilon_k}(u_k,w_k) )$ is monotonically decreasing.
    Due to Assumption \ref{ass:f}, boundedness of $(u_k)$ and $(w_k)$ in $U$ and $W$, respectively, follows.
    Boundedness of $f'(u_k)$ in $U$ then follows from boundedness of $(u_k)$ and the Lipschitz assumption on $f'$ in Assumption \ref{ass:f'}.
	Summing up the inequality above yields
	\[
		\sum_{k=1}^{\infty} \frac{\alpha}{2}\|u_{k+1}-u_k\|_U^2+\frac{\beta}{2}\|w_{k+1}-w_k\|_W^2+\gamma\int_{\Omega}\psi'_{\epsilon_k}(w_k^2)(w_{k+1}-w_k)^2\dd x < \infty
	\]
	as in \cite[Corollary 7.6]{sparse}, which implies $(u_{k+1}-u_k)\to 0$ in $U$ and $(w_{k+1}-w_k) \to 0$ in $W$.
\end{proof}

By Lemma \ref{tm:lk_bounded_nonpenalgo} and boundedness of the sequence $(u_k,w_k)$, also $(L_k)$ is bounded. \\
Let $\lambda_k \in W^*$ defined by $\braket{\lambda_k,z}_{W}=2 \gamma \int_{\Omega}\psi'_{\epsilon_k}(w_k^2)w_{k+1}z \dd x$ denote the multiplier of the smoothed $L^p$-pseudonorm.
Now, we want to pass to the limit in the optimality condition for the subproblem of the algorithm.
Due to the lack of boundedness of the multipliers $\lambda_k$, we cannot do this directly as in \cite{sparse}.
Therefore, we show that $\lambda_k$ is bounded in $ W^*$ with respect to $k$
by penalizing the constrained minimization problem \eqref{eq:minprob_algo_nonpen}.
Let us fix $k$ and take a penalty parameter $\delta>0$.
Then we consider the unconstrained problem

\begin{equation} \label{eq:minprob_algo_nonpen_pen}
     \underset{(u,w) \in U\times W}{\min}    \varphi_k(u,w)
     + \frac1{2\delta} \|(w-u)_-\|_U^2 + \frac1{2\delta} \|(w+u)_-\|_U^2.
    \end{equation}
Due to coercivity and strong convexity, problem \eqref{eq:minprob_algo_nonpen_pen} has a unique solution $(u_k^{\delta},w_k^{\delta})$.
Taking the directional derivative of the objective of problem \eqref{eq:minprob_algo_nonpen_pen} above with similar arguments as in the proof of Lemma \ref{tm:oc} leads to the optimality
condition
\begin{multline}\label{eq:oc_nonpenalg_pen}
     f'(u_k)v+\alpha(u_k^{\delta},v)_U + \beta (w_k^{\delta},z)_W + 2\gamma\int_{\Omega}\psi_{\epsilon_k}'(w_k^2)w_k^{\delta}z\dd x + L_k(u_k^{\delta}-u_k,v)_U \\+ \frac1{\delta}( (w_k^{\delta}-u_k^{\delta})_-,z-v)_U+\frac1{\delta}( (w_k^{\delta}+u_k^{\delta})_-, z+v)_U = 0 \quad \forall (v,z) \in U \times W.
\end{multline}
Let us show that $w_k^{\delta}\geq 0$ analogously to Lemma \ref{tm:wepsgeq0}.

\begin{lemma}\label{tm:wkGeq02}
    Let $(u_k^{\delta},w_k^{\delta})$ be a solution of the penalized problem \eqref{eq:minprob_algo_nonpen_pen}. Then it holds $w_k^{\delta}\geq 0$.
\end{lemma}
\begin{proof} The proof works similarly as in Lemma \ref{tm:wepsgeq0} by testing the optimality condition \eqref{eq:oc_nonpenalg_pen} with $(0,w_k^{\delta})$ and $(0,(w_k^{\delta})_+)$ and using $\int_{\Omega}\psi_{\epsilon_k}'(w_k^2)w_k^{\delta}(w_k^{\delta})_+ \dd x \le \int_{\Omega}\psi_{\epsilon_k}'(w_k^2)(w_k^{\delta})^2\dd x$.
\end{proof}

\begin{lemma}
    Let $\delta_n\searrow 0$ with $\delta_n>0$ for all $n\in \N$. Then the solutions $(u_k^n,w_k^n)$ of problem \eqref{eq:minprob_algo_nonpen_pen} with parameter $\delta_n$ are bounded independently of $k$ and $n$.
\end{lemma}
\begin{proof}
    As $(u_k^n,w_k^n)$ is a minimum of \eqref{eq:minprob_algo_nonpen_pen}, it holds
    \begin{multline*}
          f'(u_k)(u_k^n-u_k) + \frac{\alpha}{2}\|u_k^n\|_U^2 + \frac{\beta}{2}\|w_k^n\|_W^2+ \gamma \int_{\Omega}\psi'_{\epsilon_k}(w_k^2)((w_k^n)^2-w_k^2) \dd x \\+ \frac{L_k}{2}\|u_k^n-u_k\|_U^2  + \frac1{2\delta_n} \|(w_k^n-u_k^n)_-\|_U^2 + \frac1{2\delta_n} \|(w_k^n+u_k^n)_-\|_U^2
          \\
           \le \frac{\alpha}{2}\|u_{k}\|_U^2 + \frac{\beta}{2}\|w_{k}\|_W^2,
          \end{multline*}
    since $(u_{k},w_{k})$ solves \eqref{eq:minprob_algo_nonpen} and    satisfies the constraint $(u_k,w_k) \in K$.
    Due to the boundedness of $(u_k) ,(w_k)$ from Lemma \ref{lm:mon_phi_nonpen}, Assumption \ref{ass:f'}
    and the concavity of $\psi_{\epsilon_k}$, the claim follows.
\end{proof}

\begin{lemma}\label{tm:boundedness_nonpen_pen}
    Let $\delta_n\searrow 0$ and let $(u_k^n,w_k^n)$ be the corresponding solutions of \eqref{eq:minprob_algo_nonpen_pen}. Then  $\frac1{\delta_n} \|(w_k^n-u_k^n)_-\|_U $ and $ \frac1{\delta_n} \|(w_k^n+u_k^n)_-\|_U $ are bounded both with respect to $n$ and $k$.
\end{lemma}
\begin{proof}
    Testing \eqref{eq:oc_nonpenalg_pen} with $v=-(w_k^n-u_k^n)_-$ and $z=0$ yields as in Lemma \ref{tm:mult_bounded} \begin{align*}
        \frac1{\delta_n}\|(w_k^n-u_k^n)_-\|_U \leq \|f'(u_k)\|_{U} +\alpha\|u_k^n\|_U+L_k\|u_k^n-u_k\|_U.
    \end{align*} Here, we also used $w_k^n\geq 0$ from Lemma \ref{tm:wkGeq02}. The same can be done for $ \frac1{\delta_n} \|(w_k^n+u_k^n)_-\|_U $.
    Since we already know that the sequences $(u_k)$, $(u_k^n)$, $f'(u_k)$, and $(L_k)$ are bounded, this proves the claim.
\end{proof}

\begin{lemma}\label{tm:nonpenalg_conv_res}
    Let  $\delta_n \searrow 0$. Then it holds $u_k^n \to u_{k+1}$ in $U$ and $w_k^n\to w_{k+1}$ in $W$.
    Furthermore, if $\mu^1_k$ and $\mu_k^2$ are weak limits of subsequences of
    $\frac1{\delta_n} (w_k^n-u_k^n)_- $ and $ \frac1{\delta_n} (w_k^n+u_k^n)_-  $ in $U$, respectively,
    then $\mu_k^1 \in \partial I_C(w_{k+1}-u_{k+1})$ and $\mu_k^2 \in \partial I_C(w_{k+1}+u_{k+1})$
	for $C$ given by \eqref{eq_def_C}.
\end{lemma}
\begin{proof}
    Since $(u_k^n)$, $(w_k^n)$ are bounded, there is a weakly convergent subsequence also denoted by  $(u_k^n)$, $(w_k^n)$ such that $u_k^n \rightharpoonup \bar u_k$ and $w_k^n\rightharpoonup \Bar{w}_k$ for some $\bar u_k\in U$, $\Bar{w}_k\in W$. As $\frac1{\delta_n} \|(w_k^n-u_k^n)_-\|_U $ and $ \frac1{\delta_n} \|(w_k^n+u_k^n)_-\|_U $ are bounded by Lemma \ref{tm:boundedness_nonpen_pen} and $\delta_n \to 0$, it holds  $\|(w_k^n-u_k^n)_-\|_U \to 0$ and $\|(w_k^n+u_k^n)_-\|_U \to 0$. This shows that $(\bar u_k,\Bar{w}_k)\in K$ as in the proof of Theorem \ref{tm:conv_uw}.
    By optimality of $(u_k^n,w_k^n)$ in the penalized problem \eqref{eq:minprob_algo_nonpen_pen} it holds \begin{equation*}\begin{gathered}
         \varphi_k(u_{k+1},w_{k+1}) \geq \varphi_k(u_k^n,w_k^n) + \frac1{2\delta_n} \|(w_k^n-u_k^n)_-\|_U^2 + \frac1{2\delta_n} \|(w_k^n+u_k^n)_-\|_U^2.\end{gathered}\end{equation*}
         We can pass to  the limit superior for $n\to \infty$ in the inequality above to obtain
\begin{align}\begin{split}
		\label{eq:ineqConvUkWkAlgo}
         \varphi_k(u_{k+1},w_{k+1}) &\geq  \limsup_{n \to \infty}\left( \varphi_k(u_k^n,w_k^n) + \frac1{2\delta_n} \|(w_k^n-u_k^n)_-\|_U^2 + \frac1{2\delta_n} \|(w_k^n+u_k^n)_-\|_U^2 \right) \\
         &\geq \varphi_k(\bar u_k,\Bar{w}_k).
     \end{split}
    \end{align}
    Since the minimization problem \eqref{eq:minprob_algo_nonpen} is strictly convex, its solution is unique and therefore $(\bar u_k,\Bar{w}_k)=(u_{k+1},w_{k+1})$. This holds independently of the choice of the subsequence, so the whole sequence $(u_k^n,w_k^n)_n$ converges weakly to $(u_{k+1},w_{k+1})$.
    Using the lower semicontinuity of the individual terms in the definition of $\varphi_k$, we get
    \begin{multline*}
     f'(u_k)(u_{k+1}-u_k) + \liminf_{n\to\infty} \frac{\alpha}{2}\|u_k^n\|_U^2
     + \liminf_{n\to\infty}\frac{\beta}{2}\|w_k^n\|_W^2 \\
     \quad+  \liminf_{n\to\infty} \gamma \int_{\Omega}\psi'_{\epsilon_k}((w_k^2)(w_k^n)^2-w_k^2) \dd x
      + \liminf_{n\to\infty}\frac{L_k}{2}\|u_k^n -u_k\|_U^2 \\
     \ge  \varphi_k(u_{k+1},w_{k+1}) \ge  \limsup_{n \to \infty} \varphi_k(u_k^n,w_k^n),
    \end{multline*}
    where the last inequality is due to \eqref{eq:ineqConvUkWkAlgo}. Hence, the assumptions of Lemma \ref{tm:auxLimSupInf} are satisfied,
    and all the individual terms in the inequality above converge. In particular we have  $\|u_k^n\|_U^2 \to \|u_{k+1}\|_U^2$ and $\|w_k^n\|_W^2 \to \|w_{k+1}\|_W^2$. Thus, the sequences converge strongly in $U$ and $W$, respectively.\\
   	For the second statement, suppose we have $\frac1{\delta_{n_j}} (w_k^{n_j}-u_k^{n_j})_- \rightharpoonup \mu^1_k$ for $j\to\infty$.
   	Let $v\in C$, i.e. $0 \le v \in U$. By strong convergence of $(u_k^n,w_k^n)$, we get
   	\[
   	\begin{split}
   		(\mu_k^1,v-(w_{k+1} -u_{k+1}))_U &=  \lim_{j \to \infty} \frac1{\delta_{n_j}}( (w_k^{n_j}-u_k^{n_j})_-, v - (w_k^{n_j} - u_k^{n_j}))_U \\
   		&=  \lim_{j \to \infty} \frac1{\delta_{n_j}}
   		\left[ ((w_k^{n_j}-u_k^{n_j})_-, v)_U - \|(w_k^{n_j}-u_k^{n_j})_-\|_U^2 \right] \leq 0,
   	\end{split}
   	\]
    so that $\mu_k^1 \in \partial I_C(w_{k+1}-u_{k+1})$ follows. The result $\mu_k^2 \in \partial I_C(w_{k+1}+u_{k+1})$ can be obtained analogously.
\end{proof}

Now, we can pass to the limit in the optimality condition \eqref{eq:oc_nonpenalg_pen} for the penalized problem to obtain an optimality condition for the non-penalized problem.
\begin{lemma} \label{lm:ocSubproblemAlgo}
Let $(u_{k+1},w_{k+1})$ be the solution of \eqref{eq:minprob_algo_nonpen}.
Then there are  $\mu_k^1 \in \partial I_C(w_{k+1}-u_{k+1})$ and $\mu_k^2 \in \partial I_C(w_{k+1}+u_{k+1})$ such that
\begin{multline} \label{eq:oc_nonpenalg_subdiff_l2}
    f'(u_k)v+\alpha(u_{k+1},v)_U+\beta(w_{k+1},z)_W+2\gamma \int_{\Omega}\psi'_{\epsilon_k}(w_k^2)w_{k+1} z \dd x + L_k(u_{k+1}-u_k,v)_U
    \\
    + (\mu_k^1, z-v)_U + (\mu_k^2,z+v)_U = 0  \quad \forall (v,z)\in U\times W.
    \end{multline}
The multipliers $\mu_k^1$ and $\mu_k^2$ can be chosen such that the sequences $(\mu_k^1)$ and $(\mu_k^2)$ are bounded in $U$.
\end{lemma}
\begin{proof}
Due to Lemma \ref{tm:boundedness_nonpen_pen}, the sequences $(\frac1{\delta_n} (w_k^n-u_k^n)_-)$ and $( \frac1{\delta_n} (w_k^n+u_k^n)_-) $  are bounded in $U$. Hence, we can extract weakly converging subsequences with limits $\mu_k^1$ and $\mu_k^2$, respectively.
Using Lemma \ref{tm:nonpenalg_conv_res}, we can pass to the limit in  \eqref{eq:oc_nonpenalg_pen} to prove the claim.
	Boundedness for this choice of multipliers $(\mu_k^1)$ and $(\mu_k^2)$ follows from the boundedness result in Lemma \ref{tm:boundedness_nonpen_pen}.
\end{proof}

\begin{lemma}\label{tm:conv_lambda_alg}
    Let $(u_k,w_k,L_k)$ be a sequence generated by Algorithm \ref{alg:alg_nonpen}. Let $(\bar u,\Bar{w})$ be a weak limit in $U\times W$ of a subsequence $(u_{k_n},w_{k_n})$. Then it holds \begin{align*}
        2 \int_{\Omega}\psi'_{\epsilon_{k_n}}(w_{k_n}^2)w_{k_n+1}^2 \dd x \to p \int_{\Omega}|\Bar{w}|^p \dd x.
    \end{align*}
\end{lemma}
\begin{proof}
    The proof can be done as in \cite[Lemma 7.7]{sparse} using the compact embedding of $W$ into $L^2(\Omega)$ and Lemma \ref{lm:mon_phi_nonpen} for the convergence $(w_{k_n+1}-w_{k_n}) \to 0$  in $W$.
\end{proof}

\begin{lemma}\label{tm:conv_uk}
   Let $(u_k,w_k,L_k)$ be a sequence generated by Algorithm \ref{alg:alg_nonpen}.
Let $(\bar u,\Bar{w})$ be a weak limit in $U\times W$ of a subsequence $(u_{k_n},w_{k_n})$.
    Then the weak convergence is in fact strong, i.e. $u_{k_n} \to \bar u$ in $U$ for $n\to\infty$.
\end{lemma}
\begin{proof}
    Let $(u_k,w_k)$ be the solution to the subproblem \eqref{eq:minprob_algo_nonpen}.
    Then $u_k$ is the solution of
    \begin{equation}\label{eq_u_k_w_fixed}
         \underset{u\in U: (u,w_k) \in K}{\min}  f'(u_{k-1})u + \frac{\alpha}{2}\|u\|_U^2 + \frac{L_{k-1}}{2}\|u-u_{k-1}\|_U^2
    \end{equation}
    for fixed $w_k$.
    Due to strong convexity, $u_k$ is equivalently characterized by the necessary optimality conditions for \eqref{eq_u_k_w_fixed}.
    That is, $u_k$ solves \eqref{eq_u_k_w_fixed} if it solves the following variational inequality: find $u$ with $(u,w_k) \in K$ such that
    \[
     ( \alpha u + f'(u_{k-1}) + L_{k-1}(u - u_{k-1}), \ v-u) \ge0 \quad \forall v\in U: \ (v,w_k)\in K.
    \]
    Thus, 
     $u_k$ is equal to the projection of
     \[
     g_{k-1}:= - \frac1{\alpha }\left( f'(u_{k-1}) + L_{k-1}(u_k-u_{k-1} )\right)
   \]
   onto the admissible set
    $\{u\in U: |u(t,x)|\leq w_k(x) \text{ for a.e. } (t,x)\}$.
    Since $U=L^2(I\times\Omega)$, $u_k$ can be written
    as
    \begin{align}\label{eq:minprobInUForGivenW}
        u_k(t,x) = \max\left( -w_k(x),\,\min\left(w_k(x),\,  g_{k-1}(t,x)  \right)\right),
    \end{align}
    see also \cite[Example 6.24]{BonnansShapiro2000}.

    Let now $(u_{k_n},w_{k_n}) \rightharpoonup (\bar u, \bar w)$ in $U\times W$.
    Due to compact embeddings, $w_{k_n} \to \bar w$ in $L^2(\Omega)$.
    Using the result of Lemma \ref{lm:mon_phi_nonpen}, we get $(u_{k_n}-u_{k_n-1}) \to 0$.
    The boundedness of $(L_k)$ by Lemma \ref{tm:lk_bounded_nonpenalgo} and Assumption \ref{ass:f'} imply that $g_{k_n-1}$ converges strongly in $U$.
    Since $\max$ and $\min$ give rise to continuous mappings from $U$ to $U$, \eqref{eq:minprobInUForGivenW} implies that $u_{k_n}$ converges strongly in $U$.
\end{proof}

Now we can state a stationarity system for weak limit points.
Let us remark that the resulting system is  almost equal to the necessary optimality condition in Theorem \ref{thm_nec_opt_cond}.
The only difference is the inequality \eqref{eq:cond_lambda_nonpen_alg}, which is satisfied with equality in Theorem \ref{thm_nec_opt_cond}.

\begin{theorem} \label{tm:tm_limitcond_algo_nonpen}
Let $(u_k,w_k,L_k)$ be a sequence of iterates generated by Algorithm \ref{alg:alg_nonpen}, and let $\bar u$ and $\Bar{w}$ be subsequential weak limits in $U$ and $W$ of $(u_k)$ and $(w_k)$, respectively. Then there are $\Bar{\lambda}\in W^*$, $\bar{\mu}^1, \Bar{\mu}^2 \in U$ such that
\begin{align} \label{eq:oc_lim_algnonpen}
        f'(\bar u)v + \alpha (\bar u,v)_U + \beta (\Bar{w},z)_W + \gamma  \braket{\bar{\lambda},z}_W +(\Bar{\mu}^1, z-v)_U + (\Bar{\mu}^2, z+v)_U = 0
    \end{align} for all $(v,z) \in U \times W$ and
    \begin{align}\label{eq:cond_lambda_nonpen_alg}
        \braket{\Bar{\lambda},\Bar{w}}_W \geq p\int_{\Omega}|\bar{w}|^p\dd x.
    \end{align}
    Moreover, $\bar\mu^1 \in \partial I_C(\bar w-\bar u)$ and $\bar \mu^2 \in \partial I_C({\bar w + \bar u})$.
\end{theorem}
\begin{proof}
Let $(\bar u, \bar w)$ be the weak limit of a subsequence $(u_{k_n}, w_{k_n})$. By Lemma \ref{lm:ocSubproblemAlgo}, we can assume the corresponding sequence $(\mu_{k_n}^1)$ and $(\mu_{k_n}^2)$ of multipliers to be bounded. Therefore, we can a extract a subsequence still denoted by $(k_n)$ such that  $\mu_{k_n}^1\rightharpoonup \bar\mu^1$ and $\mu_{k_n}^2\rightharpoonup \bar\mu^2$ in $U$ for some $\bar\mu^1, \bar\mu^2\in U$. Let $\lambda_k\in W^*$ be defined as $\braket{\lambda_k,v}_W:= 2 \int_{\Omega}\psi'_{\epsilon_{k}}(w_{k}^2)w_{k+1} v \dd x$. By convergence of the remaining terms in \eqref{eq:oc_nonpenalg_subdiff_l2} we obtain $\lambda_{k_n} \rightharpoonup \Bar{\lambda}$ in $W^*$ for some $\Bar{\lambda}$.
Passing to the limit in \eqref{eq:oc_nonpenalg_subdiff_l2} along $({k_n})$ yields now \eqref{eq:oc_lim_algnonpen}. Next, we test \eqref{eq:oc_nonpenalg_subdiff_l2} with $(u_{{k_n}+1},w_{{k_n}+1})$
to obtain
\begin{multline}\label{eq:eqForLimInStatSyst}
f'(u_{k_n})u_{{k_n}+1}+\alpha\|u_{{k_n}+1}\|_U^2+\beta\|w_{{k_n}+1}\|^2_W + L_{k_n}(u_{{k_n}+1}-u_{k_n},u_{{k_n}+1})_U \\+ (\mu_{{k_n}}^1, w_{{k_n}+1}-u_{{k_n}+1})_U + (\mu_{{k_n}}^2, w_{{k_n}+1}+u_{{k_n}+1})_U
=
- \gamma  \braket{ \lambda_{k_n}, w_{{k_n}+1}}_W.
\end{multline}
Since $\mu_{k_n}^1$ and $\mu_{k_n}^2$ are subgradients of $I_C$ at $w_{{k_n}+1}\pm u_{{k_n}+1}$ by Lemma \ref{lm:ocSubproblemAlgo}, it holds that
\begin{align}\label{eq:for_compl_sys}    (\mu_{k_n}^1, w_{{k_n}+1}-u_{{k_n}+1})_U = 0 \hspace{5mm} \textrm{and} \hspace{5mm} (\mu_{k_n}^2, w_{{k_n}+1}+u_{{k_n}+1})_U = 0.
\end{align}
Passing to the limit inferior on the left-hand side of \eqref{eq:eqForLimInStatSyst} and using Lemma
\ref{tm:conv_lambda_alg} for its right-hand side implies \begin{align} \label{eq:tm_limitcond_algo_nonpen_aux}
    f'(\bar u)\bar u + \alpha \|\bar u\|_U^2+\beta \|\Bar{w}\|_W^2 \leq -p \, \gamma  \int_{\Omega}|\Bar{w}|^p \dd x.
\end{align}
By strong convergence of $(w_{k_n})$ and $(u_{k_n})$ in $U$ due to Lemma \ref{tm:conv_uk}, the equalities in \eqref{eq:for_compl_sys} lead to
\begin{align}\label{eq:ineq1_compl}
    (\bar{\mu}^1,\Bar{w}-\bar u)_U = 0 \hspace{5mm} \text{and} \hspace{5mm} (\bar{\mu}^2, \Bar{w}+\bar u)_U = 0.
\end{align}
Testing equation \eqref{eq:oc_lim_algnonpen} with $(\bar u,\bar w)$ therefore implies \begin{align*}
     f'(\bar u)\bar u + \alpha \|\bar u\|_U^2 + \beta \|\Bar{w}\|^2_W  =  - \gamma  \braket{\bar{\lambda},\Bar{w}}_W.
\end{align*}
Together with equation \eqref{eq:tm_limitcond_algo_nonpen_aux} this now gives \begin{align*}
            \braket{\Bar{\lambda},\Bar{w}}_W \geq p\int_{\Omega}|\bar{w}|^p\dd x.
\end{align*}
\end{proof}
One can observe that the stationarity system derived here is weaker than the necessary optimality condition from Theorem \ref{thm_nec_opt_cond}, as we only obtain an inequality in \eqref{eq:cond_lambda_nonpen_alg}. The reason for this is that we only have weak convergence $w_k\rightharpoonup \bar w$ in $W$.
In order to pass to the limit in \eqref{eq:eqForLimInStatSyst} to obtain equality in \eqref{eq:tm_limitcond_algo_nonpen_aux}
we would require strong convergence of the sequence $(w_k)$, which is necessary to prove Lemma \ref{tm:theory_lambda_conv}.

\begin{theorem}
    Let  $(\bar u, \Bar{w})$ be a weak limit point of the algorithm such that the map $x\mapsto \| f'(\bar u)(\cdot,x)\|_{L^1(I)}$ is in $L^r(\Omega)$ with $r$ as in  Lemma \ref{tm:aux_linf}. Then it holds $\Bar{w}\in L^{\infty}(\Omega)$ and $\bar u\in L^{\infty}(I\times \Omega)$.
\end{theorem}
\begin{proof}
This can be shown similarly as in Theorem \ref{tm:linf_reg} by testing the optimality condition \eqref{eq:oc_nonpenalg_subdiff_l2} of the subproblem with $(\bar u-\bar u_n, \bar w-\bar w_n)$ defined as in the proof of Theorem  \ref{tm:linf_reg}.
\end{proof}

\section{Numerical Results}\label{sec:numRes}
In this section, we consider some numerical examples to test Algorithm \ref{alg:alg_nonpen}.
For the implementation of the algorithm, we use $P_1$-finite elements
to construct the discrete spaces $U_h$ and $W_h$ on $I\times\Omega$ and $\Omega$.
The domain $\Omega$ is triangulated with some shape regular mesh with $N$ nodes and mesh-size $h$, the interval $I$ is divided by $M$ equidistant nodes into intervals of equal length.
Then $W_h$ is defined as the $P_1$-space on the triangulation of $\Omega$, whereas $U_h$ is the tensor-product space on the subdivisions of $I$ and $\Omega$.
It follows $\dim(U_h) = M \cdot N$ and $\dim(W_n) = N$.
Let $(\phi_i)$ and $(\phi_{i,j})$ denote the basis functions of $W_h$ and $U_h$.
Then it holds $(u,w)=(\sum_{i,j}u_{i,j}\varphi_{i,j}, \sum_i w_i\varphi_i)\in K \cap (U_h\times W_h)$ if and only if $|u_{i,j}|\leq w_i$ for all $i,j$.

For solving subproblem \eqref{eq:minprob_algo_nonpen} in each step of the iterative scheme, we employ a semismooth Newton method as presented in \cite{21ss_kanzow}. To apply this method, the discretized optimality condition of \eqref{eq:minprob_algo_nonpen} is rewritten as a system of equations in the variables $u\in \R^{M \cdot N}, w\in \R^N$ and $\mu^1,\mu^2 \in \R^{M \cdot N}$ using the Fischer-Burmeister function.  Here, 
$\mu^1,\mu^2$ are the Lagrange multipliers for the constraint $(u,w)\in K$. The system is then solved
with a globalized semismooth Newton method, where we switch to the gradient step of a merit function whenever the Newton iterate is not well defined or does not satisfy a certain descent condition.\\
In every step of the semismooth Newton method we solve a linear system to obtain the search direction for the new iterates $u_{k+1}, w_{k+1}, \mu^1_{k+1}$ and $ \mu^2_{k+1}$. The matrix of the linear system is a block matrix consisting mostly of sparse matrices.
Only the block of size $N\times N$ that corresponds to the discretization of the $W$ inner product
is not sparse but dense.
If we had not introduced the auxiliary function $w$ and had directly worked with \eqref{eq_minprob_u_in_Hs},
the resulting system would have a dense block of size $(M \cdot N) \times (M \cdot N)$.

In order to find a parameter $L_k$ such that the descent condition \eqref{eq:alg_nonpen_fcond} is satisfied, we employ a procedure similar to Armijo line-search in nonlinear optimization.
We solve \eqref{eq:minprob_algo_nonpen} for $L_k = L b^l$, where $l= 0,1,\dots$ is increased until \eqref{eq:alg_nonpen_fcond} is satisfied.
In the following examples we choose the parameters $L=1$ and $b=2$ in this loop.

Next, we consider two numerical examples where $f$ is chosen to be a tracking type functional of the form $f(u)=\frac{a}{2}\|u-u_d\|_U^2$ for some given function $u_d\in U$ and $a>0$. In these examples we choose the space $W=\Tilde{H}^s(\Omega)$.
For computing the $\tilde H^s(\Omega)$-stiffness matrix, we follow \cite{control1dim} in the one-dimensional and \cite{fe2d} in the two-dimensional case.

\subsection{Example in 1d}\label{sec:1dExample}
First, we consider a one-dimensional example with $\Omega := (-1,1)$, $I=(0, \frac12)$ and choose \begin{equation*}
    u_d := t \sin(1.5 (x-1) ) \hspace{5mm} \text{ and } \hspace{5mm} a=25.
\end{equation*} Furthermore, let $\alpha=2, \beta =0.2 , p=0.01$ and $s=0.1$ We compare the solutions for different values of $\gamma$. One can observe in Table \ref{tab:support} that the support of the solution $w$ decreases with increasing $\gamma$.
In Table \ref{tab:convergence}, the errors $\|u_k-u_\reference\|_U$, $\|w_k-w_\reference\|_W$ and $|\Phi_0(u_k,w_k)-\Phi_0(u_\reference,w_\reference)|$ are compared after different numbers of iterations. Here, the reference solution $(u_\reference,w_\reference)$ is the final iterate after 59 iterations when the stopping criterion is satisfied.
Table \ref{tab:meshSize} shows that the distance of solutions computed with different mesh sizes to a reference solution decreases for decreasing mesh size, where the
reference solution is computed on a finer grid. Here, we work with the same number of steps in time and space, so that $M=N$. Figure \ref{fig:1dExample} depicts the computed solutions $u$ and $w$ for this example.

\begin{figure}[h]
	\begin{subfigure}{0.3\textwidth}
		\includegraphics[scale=0.31]{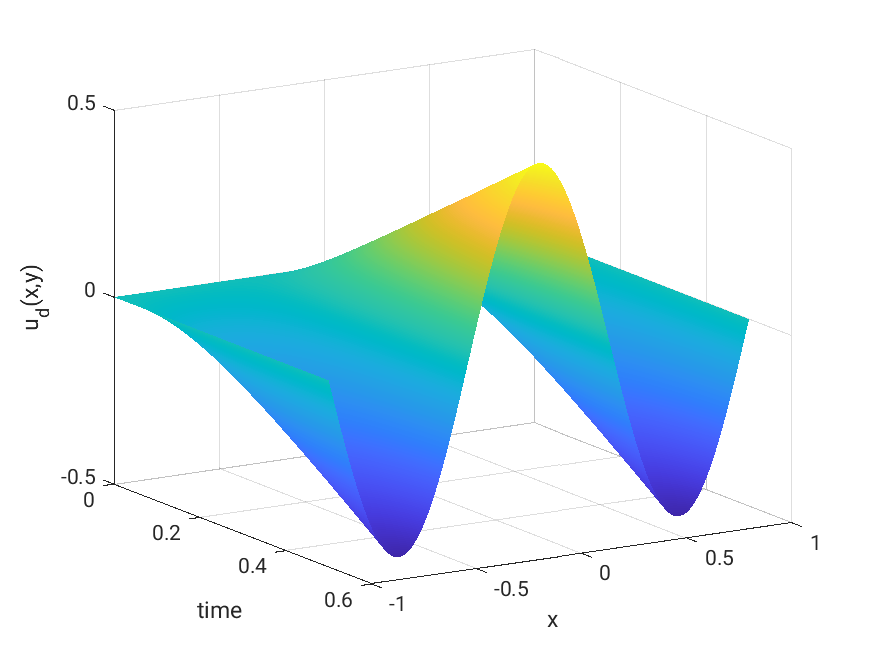}
		\caption{$u_d$.}
	\end{subfigure}
	\begin{subfigure}{0.3\textwidth}
		\includegraphics[scale=0.31]{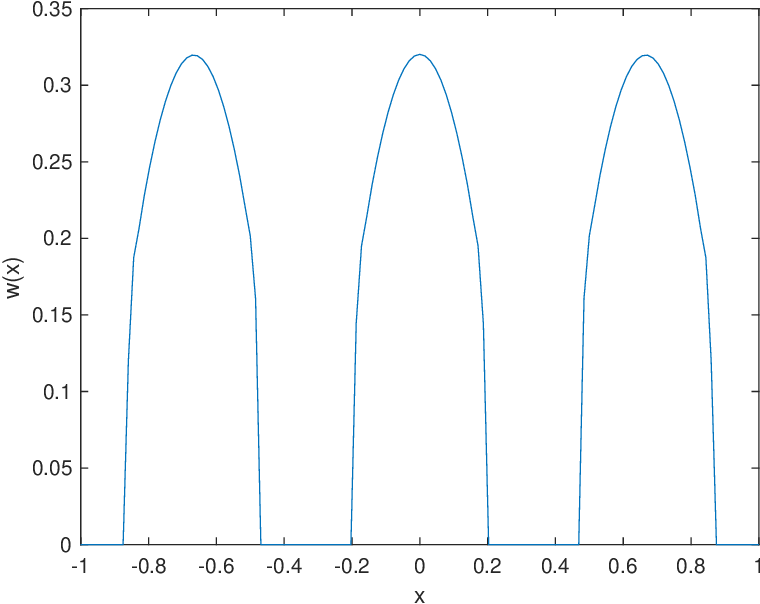}
		\caption{Solution $w_\reference$.}
	\end{subfigure}
	\begin{subfigure}{0.3\textwidth}
		\includegraphics[scale=0.31]{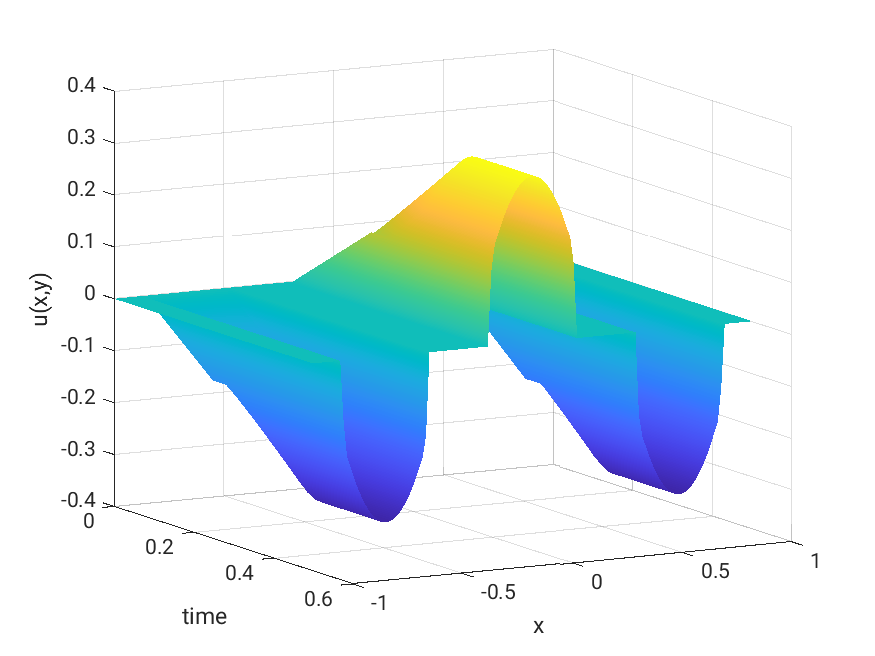}
		\caption{Solution $u_\reference$.}
	\end{subfigure}
    \caption{Section \ref{sec:1dExample}: Plot of $u_d$ and the computed solutions $w_\reference$ and $u_\reference$ for $\gamma = 5$. }\label{fig:1dExample}
\end{figure}

\begin{table}[htb]
\sisetup{table-alignment-mode = format}
\begin{center}
\begin{tabular}{S[table-format = 2]S[table-format = 2.1] }
\toprule
$\gamma$&{\textrm{\% of }$\{u(t,x)=0\}$ in $I\times\Omega$}\\
\midrule
0& 0\% \\
1 & 18.5\% \\
5 & 43\% \\
10 & 68\% \\
\bottomrule
\end{tabular}
\caption{Section \ref{sec:1dExample}: Percentage of the domain where $u$ vanishes for different values of $\gamma$. }\label{tab:support}
\end{center}
\end{table}

\begin{table}[htb]
\sisetup{table-alignment-mode = format}
\begin{center}
\begin{tabular}{S[table-format = 2]S[table-format = 1.4e2]
S[table-format=1.4e2]S[table-format=1.5]S[table-format=1.4] }
\toprule
$k$& {$\|u_k-u_\reference\|_U$} & {$\|w_k-w_\reference\|_W$} & {$|\Phi_0(u_k,w_k)-\Phi_0(u_\reference,w_\reference)|$} & {$\Phi(u_k,w_k)$}\\
\midrule
10 & 5.5871e-3 & 1.6573e-2 & 0.15402 & 1.9887\\
20 & 4.1156e-3 & 1.2061e-2 & 0.11178 & 1.9465\\
30 & 3.657e-3 & 1.0559e-2 & 0.08259 & 1.9173\\
40 & 2.5632e-3 & 7.198e-3 & 0.05709 & 1.8918\\
50 & 5.126e-06 & 9.6418e-06 & 0.0233 & 1.858\\
55 & 1.3119e-07 & 1.9858e-08 & 0.00835 & 1.8431\\
59 &  {-}  & {-}  &  {-}  & 1.8347 \\
\bottomrule
\end{tabular}
\caption{Section \ref{sec:1dExample}: Errors for different numbers of iterations  for $\gamma=1$. }\label{tab:convergence}
\end{center}
\end{table}

\begin{table}[htb]
\sisetup{table-alignment-mode = format,
table-number-alignment = left}
\begin{center}
\begin{tabular}{S[table-format = 3]S[table-format = 1.2e2]
S[table-format=1.2e2] }
\toprule
$N$& {$\|u_h-u_{\reference}\|_U$}  & {$\|w_N-w_{\reference}\|_{L^2(\Omega)}$}\\
\midrule
33 & 4.2395e-04 & 2.2583e-04 \\
65 & 8.6495e-05 & 3.2757e-05 \\
125 & 4.5303e-05 & 1.2886e-05 \\
\bottomrule
\end{tabular}
\caption{Section \ref{sec:1dExample}: Errors of the final iterates $(u_N, w_N)$ for a mesh with $N$ nodes in space compared to the reference solution $(u_\reference, w_\reference)$ for $N=257$. }\label{tab:meshSize}
\end{center}
\end{table}

\subsection{Example in 2d}\label{sec:2dExample}
We now consider the two-dimensional domain $\Omega = (-1,1)\times(-1,1)$, the time interval $I=(0,0.3)$ and choose $\alpha=2, \beta = 0.2, \gamma = 1$ and $s=0.1$ for
\[
u_d(t,x,y) = 5t\cdot \max(x^2,y^2).
\]
The number of nodes is $N=1135$ on $\Omega$ and $M=25$ on the time interval $I$.
We computed solutions for different values of $p$ from the interval $(0,1]$.
We report the size of their support in Table \ref{tab:supportP}.
As can be seen, the value $p =0.3$ leads to solutions with smallest support, whereas $p=1$ clearly results in less sparse solutions.
This behaviour of the sparsity is in line with the observations made in \cite{87huberApprox}. 
The sparsity inducing effect of the $L^p$-pseudo-norm can also be observed in Figure \ref{fig:2dExample}.

\begin{table}[htb]
\sisetup{table-alignment-mode = format}
\begin{center}
\begin{tabular}{S[table-format = 1.2]S[table-format = 2] }
\toprule
$p$&{\textrm{\% of }$\{u(x)=0\}$ in $\Omega$}\\
\toprule
1 & 13\% \\
0.9 & 20\% \\
0.7 & 31\% \\
0.3 & 37\% \\
0.1 & 32\% \\
0.05 & 25\% \\
\bottomrule
\end{tabular}
\caption{Section \ref{sec:2dExample}: Percentage of the parts where $u$ vanishes on the domain $\Omega$ for different values of $p$. }\label{tab:supportP}\end{center}
\end{table}

\begin{figure}[h]
	\begin{subfigure}{0.33\textwidth}
		\includegraphics[scale=0.33]{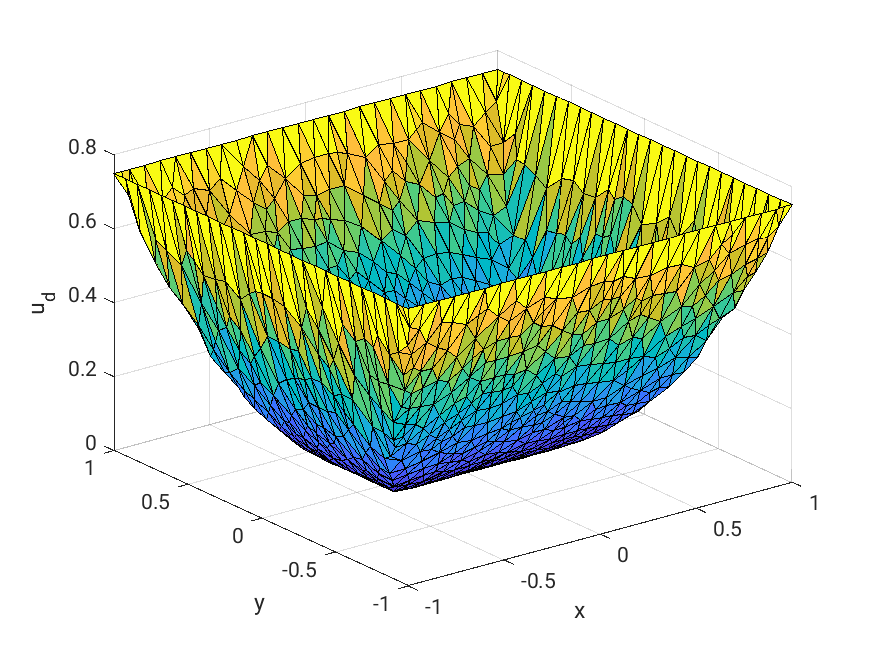}
		\setcounter{subfigure}{0}
		\caption{$u_d$ at time $T/2$.}
	\end{subfigure}
	\begin{subfigure}{0.33\textwidth}
		\includegraphics[scale=0.33]{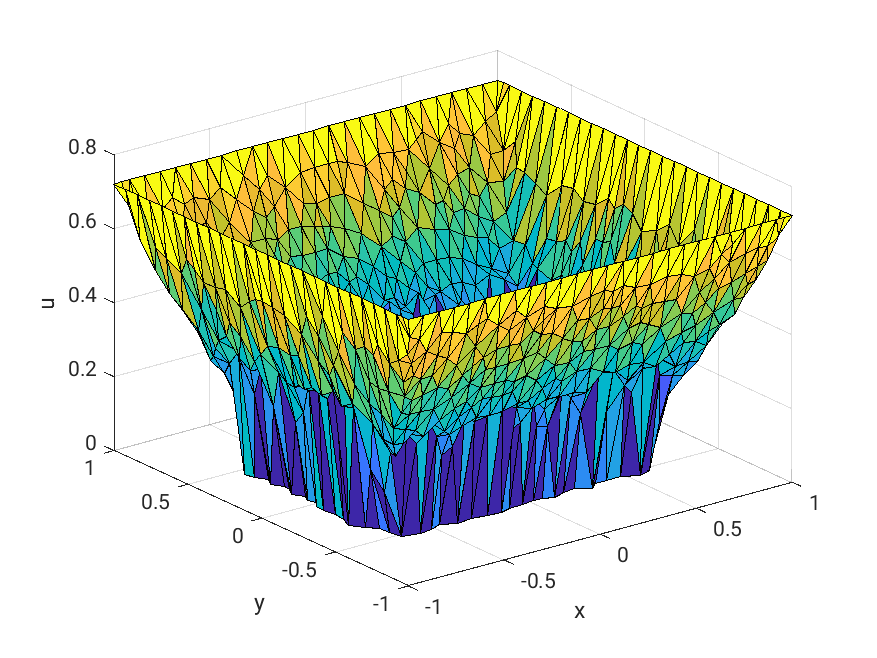}
		\setcounter{subfigure}{2}
		\caption{$u$ at time $T/2$.}
	\end{subfigure}
	\begin{subfigure}{0.33\textwidth} \vspace{-5cm}
		\includegraphics[scale=0.33]{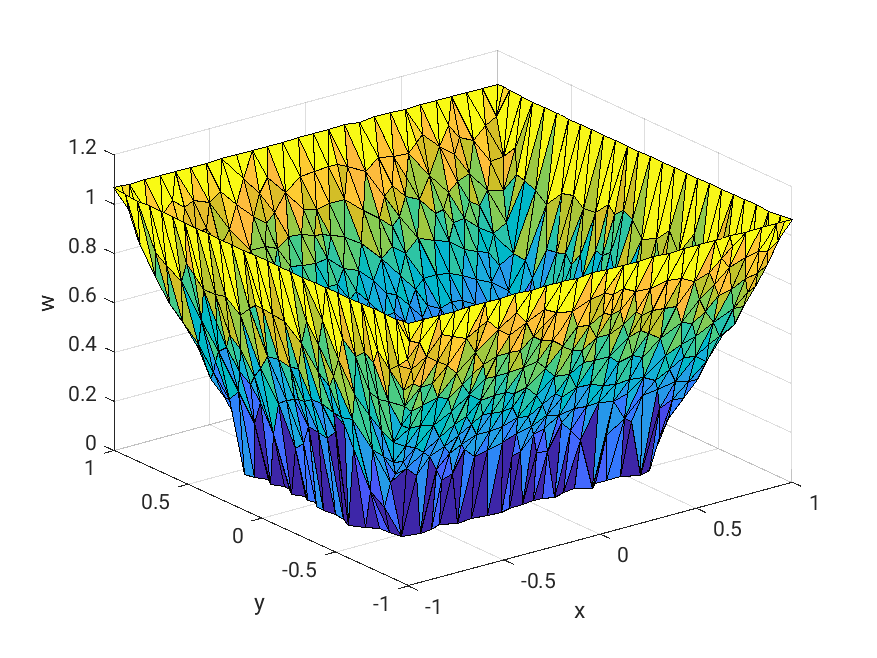}
		\setcounter{subfigure}{4}
		\caption{ $w$.}
	\end{subfigure} \\
	\begin{subfigure}{0.33\textwidth}
		\includegraphics[scale=0.33]{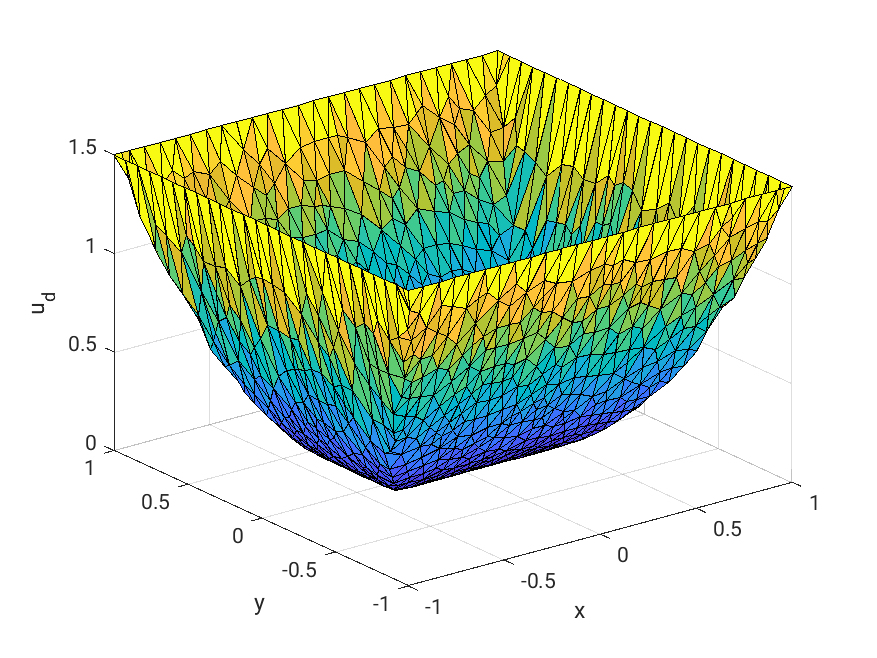}
		\setcounter{subfigure}{1}
		\caption{$u_d$ at time $T$.}
	\end{subfigure}
	\begin{subfigure}{0.33\textwidth}
		\includegraphics[scale=0.33]{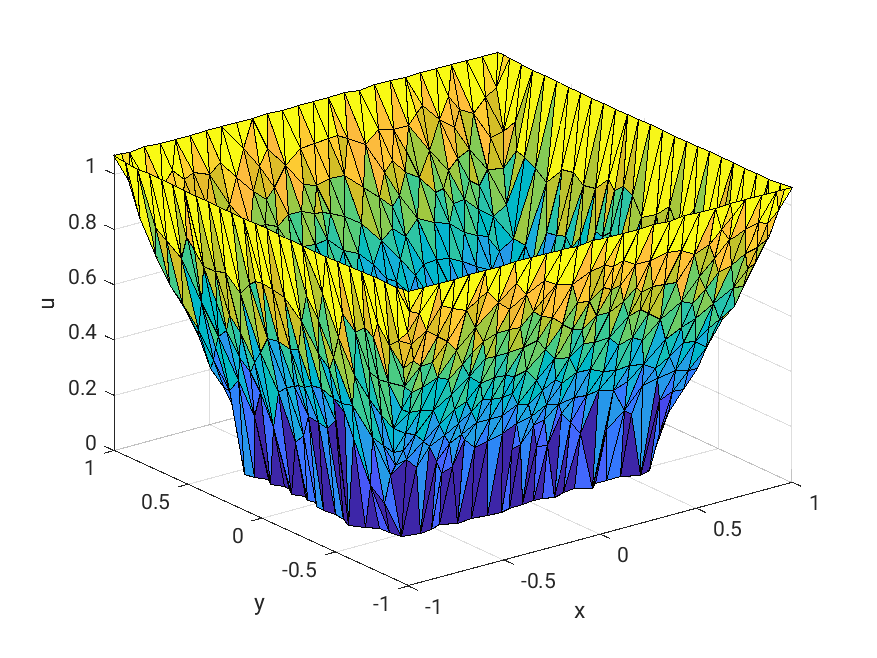}
		\setcounter{subfigure}{3}
		\caption{$u$ at time $T$.}
	\end{subfigure}
	\caption{Section \ref{sec:2dExample}: Plots of $u_d$ and numerical solutions $u$ and $w$ for $p=0.3$.}\label{fig:2dExample}
\end{figure}

\bibliographystyle{tfs}
\bibliography{lit.bib}

\end{document}